\documentclass[12pt,reqno]{amsart}

\usepackage{times}
\usepackage[T1]{fontenc}
\usepackage{mathrsfs}
\usepackage{latexsym}
\usepackage[dvips]{graphics}
\usepackage{sidecap}
\usepackage{epsfig}
\usepackage{amsmath,amsfonts,amsthm,amssymb,amscd}
\input amssym.def
\input amssym.tex
\usepackage{color}
\usepackage{hyperref}
\usepackage{url}
\usepackage{changepage}

\usepackage{pgf,tikz}
\usetikzlibrary{arrows}
\usetikzlibrary{decorations.markings}

\newcommand{\burl}[1]{\textcolor{blue}{\url{#1}}}
\newcommand{\emaillink}[1]{\textcolor{blue}{\href{mailto:#1}{#1}}}

\usepackage[top=1in, bottom=1in, left=1in, right=1in]{geometry}
\usepackage{amsmath,amssymb}

\numberwithin{equation}{section}

\newtheorem{theorem}{Theorem}[section]
\newtheorem{lemma}[theorem]{Lemma}
\newtheorem{proposition}[theorem]{Proposition}
\newtheorem{corollary}[theorem]{Corollary}

\newtheorem{conjecture}[theorem]{Conjecture}

\newtheorem{definition}[theorem]{Definition}
\newtheorem{example}[theorem]{Example}

\newcommand{\x}{\mathbf{x}}
\newcommand{\y}{\mathbf{y}}
\newcommand{\kk}{\mathbf{k}}
\newcommand{\Z}{\mathbb{Z}}

\newcommand{\R}{\mathbb{R}}

\newcommand{\N}{\mathbb{N}}

\begin{document}

\title{Sets Characterized by Missing Sums and Differences in Dilating Polytopes}

\author[Do]{Thao Do}
\email{\emaillink{thao.do@stonybrook.edu}}
\address{Mathematics Department, Stony Brook University, Stony Brook, NY, 11794}

\author[Kulkarni]{Archit Kulkarni}
\email{\emaillink{auk@andrew.cmu.edu}}
\address{Department of Mathematical Sciences, Carnegie Mellon University, Pittsburgh, PA 15213}

\author[Miller]{Steven J. Miller}
\email{\emaillink{sjm1@williams.edu}, \emaillink{Steven.Miller.MC.96@aya.yale.edu}}
\address{Department of Mathematics and Statistics, Williams College, Williamstown, MA 01267}

\author[Moon]{David Moon}
\email{\emaillink{dm7@williams.edu}}
\address{Department of Mathematics \& Statistics, Williams College, Williamstown, MA 01267}

\author[Wellens]{Jake Wellens}
\email{\emaillink{jwellens@caltech.edu}}
\address{Department of Mathematics, California Institute of Technology, Pasadena, CA, 91125}

\author[Wilcox]{James Wilcox}
\email{\emaillink{wilcoxjay@gmail.edu}}
\address{Department of Mathematics and Statistics, Williams College, Williamstown, MA 01267}



\date{\today}

\begin{abstract}
A sum-dominant set is a finite set $A$ of integers such that $|A+A| > |A-A|$.  As a typical pair of elements contributes one sum and two differences, we expect sum-dominant sets to be rare in some sense.  In 2006, however, Martin and O'Bryant showed that the proportion of sum-dominant subsets of $\{0,\dots,n\}$ is bounded below by a positive constant as $n\to\infty$.  Hegarty then extended their work and showed that for any prescribed $s,d\in\mathbb{N}_0$, the proportion $\rho^{s,d}_n$ of subsets of $\{0,\dots,n\}$ that are missing exactly $s$ sums in $\{0,\dots,2n\}$ and exactly $2d$ differences in $\{-n,\dots,n\}$ also remains positive in the limit.

We consider the following question: are such sets, characterized by their sums and differences, similarly ubiquitous in higher dimensional spaces?  We generalize the integers in a growing interval to the lattice points in a dilating polytope.  Specifically, let $P$ be a polytope in $\mathbb{R}^D$ with vertices in $\mathbb{Z}^D$, and let $\rho_n^{s,d}$ now denote the proportion of subsets of $L(nP)$ that are missing exactly $s$ sums in $L(nP)+L(nP)$ and exactly $2d$ differences in $L(nP)-L(nP)$.  As it turns out, the geometry of $P$ has a significant effect on the limiting behavior of $\rho_n^{s,d}$.  We define a geometric characteristic of polytopes called local point symmetry, and show that $\rho_n^{s,d}$ is bounded below by a positive constant as $n\to\infty$ if and only if $P$ is locally point symmetric.  We further show that the proportion of subsets in $L(nP)$ that are missing exactly $s$ sums and at least $2d$ differences remains positive in the limit, independent of the geometry of $P$.  A direct corollary of these results is that if $P$ is additionally point symmetric, the proportion of sum-dominant subsets of $L(nP)$ also remains positive in the limit.

\end{abstract}

\thanks{This research was conducted as part of the 2013 SMALL REU program at Williams College and was partially supported funded by NSF grant DMS0850577 and Williams College; the third named author was also partially supported by NSF grant DMS1265673. We would like to thank our colleagues from the Williams College 2013 SMALL REU program, especially Frank Morgan, as well as Kevin O'Bryant for helpful conversations.}

\maketitle

\setcounter{equation}{0}

\tableofcontents

\section{Introduction}

Given a finite set $A\subset\Z$, we define  the sumset $A+A$ and the difference set $A-A$ by
\begin{align} A+A &\ =\ \{a_1 + a_2 : a_1, a_2 \in A\}, \nonumber\\
A-A &\ =\ \{a_1 - a_2 : a_1, a_2 \in A\}.
\end{align}
It is natural to compare the sizes of $A+A$ and $A-A$ as we vary $A$ over a family of sets.  As addition is commutative while subtraction is not, a pair of distinct elements $a_1, a_2 \in A$ generates two differences $a_1-a_2$ and $a_2-a_1$ but only one sum $a_1+a_2$.  We thus expect that most of the time, the size of the difference set is greater than that of the sumset---that is, we expect most sets $A$ to be \emph{difference-dominant}.  It is possible, however, to construct sets whose sumsets have more elements than their difference sets.  Such sets are called \emph{sum-dominant} or \emph{More Sums Than Differences} (MSTD) sets.  The first example of an MSTD set was discovered by Conway in the 1960s: $\{0,2,3,4,7,11,12,14\}$.  A set whose sumset has the same number of elements as its difference set is called \emph{balanced}.

We briefly review some of the key results in the field.  In 2006, Martin and O'Bryant \cite{MO} showed that not only do MSTD sets exist, but there exist many of them in some sense.  In particular, they proved that the proportion $\rho_n$ of subsets of $\{0,1,\dots,n\}$ that are MSTD is bounded below by a positive constant as $n\to\infty$.  They show that similar results hold as well for balanced and difference-dominant sets.  Hegarty \cite{He} then extended their work and showed that for any $s,d \in \N_0$, the proportion $\rho_n^{s,d}$ of subsets $A \subset \{0,1,\dots,n\}$ satisfying
\begin{align}
|\{0,1,\dots,2n\} \setminus (A+A)|\ =\ s,\ \ \ \ |\{-n,-n+1,\dots,n-1,n\} \setminus (A-A)|\ =\ 2d
\end{align}
also remains bounded below by a positive constant in the limit. Later, in 2010, Zhao \cite{Z} showed that both $\rho_n$ and $\rho_n^{s,d}$ converge as $n \to \infty$, with $\rho_n$ approaching a limit $\rho\simeq 4.5 \times 10^{-4}$.


This previous work explored the behavior of sums and differences of sets in the one-dimensional lattice $\Z$.  In particular it was observed that sum-dominant, balanced, and difference-dominant sets, as well as sets with even greater constraints on missing sums and differences, are all surprisingly ubiquitous on the line.  A natural question arises: are such sets similarly common in other spaces?

In this paper, we extend the theory to sets in higher dimensional lattices, namely $\Z^D$ for any $D>0$.\footnote{See \cite{DKMMW} for another generalization to sums and differences of correlated random pairs of sets in $\Z$.}  Interesting new features and complications arise in higher dimensions.  Whereas on the line it is natural to consider subsets of the integers in a growing interval, in higher dimensions we can begin to consider different geometries for our overall subset region.  A natural high-dimensional analogue of the interval is a convex polytope.  We examine in particular the additive behavior of the lattice points in an arbitrary dilating $D$-dimensional convex polytope with lattice point vertices.

Let $P$ be a convex polytope in $\mathbb{R}^D$ with vertices in $\mathbb{Z}^D$.  For any set $S\subset\R^D$, let $L(S)$ denote the set of lattice points contained in $S$; that is, $L(S) = S \cap \Z^D$.  Furthermore, let $nS$ denote the dilation of $S$ by a factor of $n$ about the origin.  In the spirit of Hegarty, we focus our attention to the proportion $\rho_n^{s,d}$ of subsets $A\subset L(nP)$ such that
\begin{align}
|(L(nP)+L(nP)) \setminus (A+A)|\ =\ s,\ \ \ |(L(nP)-L(nP)) \setminus (A-A)|\ =\ 2d,
\end{align}
for any prescribed $s,d \in \N_0$.  In this paper we assume that $P$ is fixed, and revert to the more informal description that such subsets $A$ are missing exactly $s$ sums and missing exactly $2d$ differences.  Studying missing sums and differences rather than the number of sums and differences is the natural generalization of the 1-dimensional results, which we discuss at the end of this section and in Section \ref{sec:conclusion}.



The geometry of $P$ has a significant effect on the limiting behavior of $\rho_n^{s,d}$.  Before we state our main results, we introduce some terminology that helps us distinguish between polytopes.

\begin{definition}
Let $P$ be a convex polytope.  Vertices $\mathbf{u}$ and $\mathbf{v}$ of $P$ are \emph{strictly antipodal} if there exist parallel supporting hyperplanes, $H_1$ and $H_2$, of $P$ such that $H_1 \cap P = \{\mathbf{u}\}$ and $H_2 \cap P = \{\mathbf{v}\}$.
\end{definition}

\begin{definition}
Given a vertex $\mathbf{v}$ of $P$, the \emph{supporting cone} $C(\mathbf{v})$ at $\mathbf{v}$ is the set
\begin{align}
\mathbf{v} + \bigcup_{\lambda\ge 0} \lambda(P-\mathbf{v}).
\end{align}
Equivalently, $C(\mathbf{v})$ is the convex hull of the half-lines formed by extending the edges of $P$ at $v$.
\end{definition}

\begin{definition}
A polytope $P$ is \emph{point symmetric} if there exists a point $\mathbf{x}$ such that $P = \mathbf{x} - P$.
\end{definition}

\begin{definition}
A convex polytope $P$ with $m$ vertices is \emph{locally point symmetric} if its vertices can be partitioned into $m/2$ pairs of strictly antipodal vertices such that for each pair $\{\mathbf{u}, \mathbf{v}\}$,
\begin{align}
C(\mathbf{u}) - \mathbf{u}\ =\ \mathbf{v} - C(\mathbf{v}).
\end{align}
\end{definition}

Note we subtract the vertex above (in $C(\mathbf{u}) - \mathbf{u}$ and $\mathbf{v} - C(\mathbf{v})$) so that the supporting cones are standardized with their apexes at the origin.

\begin{example}
Any point symmetric polytope is locally point symmetric.
\end{example}

\begin{example}
Consider the hexagon $ABCDEF$ in Figure \ref{fig:lpsHex}, where $A$ and $D$, $B$ and $E$, and $C$ and $F$ form pairs of strictly antipodal vertices.  As $\overline{AB}$ and $\overline{DE}$, $\overline{BC}$ and $\overline{EF}$, and $\overline{CD}$ and $\overline{FA}$ form pairs of parallel edges, $ABCDEF$ is locally point symmetric.
\end{example}

\begin{figure}[h]
\centering
\definecolor{ffffff}{rgb}{1,1,1}
\begin{tikzpicture}[line cap=round,line join=round,>=triangle 45,x=1.0cm,y=1.0cm,scale=0.5]
\draw (1.5,3) -- (5,3) -- (9,-3) -- (7,-5) -- (1,-5) -- (-2,-.5) -- cycle;
\begin{scriptsize}
\draw[color=black] (1.3,3.4) node {$A$};
\draw[color=black] (5.2,3.4) node {$B$};
\draw[color=black] (9.5,-3) node {$C$};
\draw[color=black] (7.2,-5.4) node {$D$};
\draw[color=black] (0.8,-5.4) node {$E$};
\draw[color=black] (-2.5,-.5) node {$F$};
\end{scriptsize}
\end{tikzpicture}
\vspace{-0.2cm}
\caption{\small \label{fig:lpsHex} A locally point symmetric hexagon.}
\end{figure}
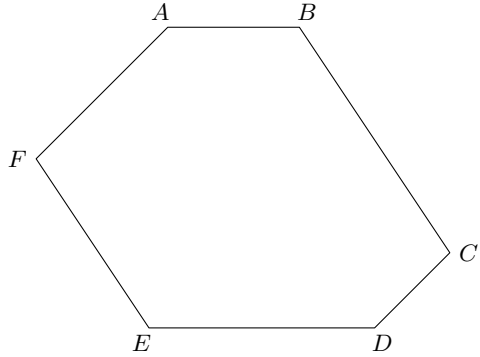

As it turns out, whether $P$ has local point symmetry determines whether $\rho_n^{s,d}$ remains positive in the limit.  We prove the following result.

\begin{theorem} \label{thm:1}
Let $P$ be a convex polytope in $\R^D$ with vertices in $\Z^D$, and let $s,d \in \N_0$ be given.  There exists a constant $c_{s,d} > 0$ such that, for sufficiently large $n$, at least $c_{s,d}\cdot 2^{|L(nP)|}$ of the subsets of $L(nP)$ have exactly $s$ missing sums and exactly $2d$ missing differences if and only if $P$ is locally point symmetric.
\end{theorem}

We restrict ourselves to polytopes with lattice point vertices because, as we will see, this allows us to exploit results in the one-dimensional case.  The main idea behind Theorem \ref{thm:1} is that convex polytopes without local point symmetry (and these constitute the vast majority of convex polytopes) have many uniquely formed differences as they dilate by $n$.  That is, there exist many differences $\mathbf{k} \in L(nP) - L(nP)$ each of for which there exists a unique pair of elements $\mathbf{p},\mathbf{q}\in L(nP)$ that satisfies $\mathbf{k}  = \mathbf{p} - \mathbf{q}$.  This makes it vanishingly unlikely as $n$ grows that there is a constant number of missing differences in the region $L(nP)-L(nP)$.


On the other hand, we can weaken our condition on the number of missing differences and obtain a positive proportion in the limit, independent of the geometry of $P$.

\begin{theorem} \label{thm:2}
Let $P$ be a convex polytope in $\R^D$ with vertices in $\Z^D$, and let $s,d\in\N_0$ be given.  There exists a constant $c_{s,d} > 0$ such that, for sufficiently large $n$, at least $c_{s,d}\cdot 2^{|L(nP)|}$ of the subsets of $L(nP)$ have exactly $s$ missing sums and at least $2d$ missing differences.
\end{theorem}

As mentioned above, studying missing sums and differences provides a more natural framework in which to consider the additive behavior of high-dimensional sets.  If $D = 1$, and therefore $P$ is an interval, then setting $2d > s$ in the theorems above implies a positive lower bound on the proportion of MSTD subsets of $L(nP)$ as $n\to\infty$; this is Hegarty's generalization \cite{He} of the results of Martin and O'Bryant \cite{MO}.  The reason for this is that the overall set region $L(nP)$ is itself balanced, and thus having more sums than differences is equivalent to having more missing differences than missing sums. This is occasionally true in higher dimensions as well.  For example, consider subsets $A$ of the square $S_n := \{ (x,y) : x,y\in\{0,\dots,n\} \}$.  We see that $A+A$ lives in the square $S_n+S_n = \{ (x,y) : x,y\in\{0,\dots,2n\} \}$ and $A-A$ lives in the square $S_n - S_n = \{ (x,y) : x,y\in\{-n,\dots,n\}\}$, both regions having $(2n+1)^2$ elements.


As our polytope $P$ varies, however, it is much more typical that the difference set region $L(nP)-L(nP)$ is larger than the sumset region $L(nP)+L(nP)$.  If we now consider subsets $A$ of the triangle $T_n := \{ (x,y) \in \Z^2 : x\ge 0, y\ge 0, x+y \le n \}$, then $A+A$ lives inside $T_n + T_n$, which has $2n^2+3n+1$ elements, while $A-A$ lives inside $T_n - T_n$, which has $3n^2+3n+1$ elements; see Figure \ref{fig:triangle}.  Observe that $|T_n -T_n| - |T_n + T_n| = n^2$.  Since we fix the number $2d$ of missing differences independently of $n$, any $A\subset T_n$ that is missing exactly $2d$ differences will, for sufficiently large $n$, always result in a difference set $A-A$ that has more elements than is even possible in the sumset $A+A$.

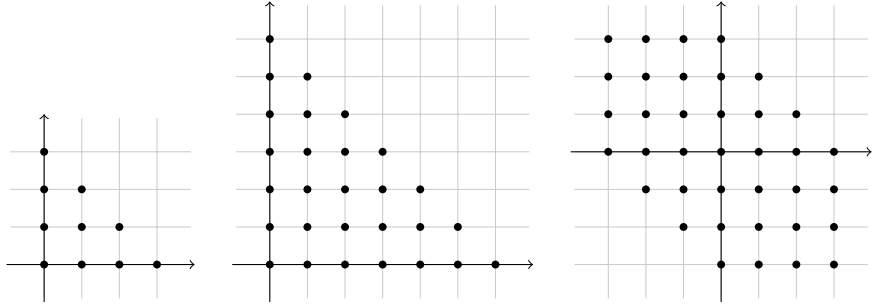
\begin{figure}[b]
\centering
\begin{tikzpicture}[scale=0.5]


\draw [step=1,thin,gray!40] (-0.9,-0.9) grid (3.9,3.9);

\draw [->] (-1,0) -- (4,0);
\draw [->] (0,-1) -- (0,4);

\foreach \x in {0,1,...,3}
	\foreach \y in {0,1,...,3}
    {
    	\pgfmathparse{int(\x+\y)};
    	\ifnum \pgfmathresult>3
       		{}
        \else
        	{\fill (\x,\y) circle[radius=3pt];}
        \fi
    }



\draw [step=1,thin,gray!40] (5.1,-0.9) grid (12.9,6.9);

\draw [->] (5,0) -- (13,0);
\draw [->] (6,-1) -- (6,7);

\foreach \x in {0,1,...,6}
	\foreach \y in {0,1,...,6}
    {
    	\pgfmathparse{int(\x+\y)};
    	\ifnum \pgfmathresult>6
       		{}
        \else
        	{\fill (\x+6,\y) circle[radius=3pt];}
        \fi
    }



\draw [step=1,thin,gray!40] (14.1,-0.9) grid (21.9,6.9);

\draw [->] (14,3) -- (22,3);
\draw [->] (18,-1) -- (18,7);

\foreach \x in {-3,...,3}
	\foreach \y in {-3,...,3}
    {

    	\pgfmathparse{int(\x+\y)}
        \ifnum \pgfmathresult>3
        {}
        \else
        {
            \pgfmathparse{int(\x+\y)}
            \ifnum \pgfmathresult<-3
            {}
            \else
            {
            	\fill (\x+18,\y+3) circle[radius=3pt];
            }
            \fi
        }
        \fi

    }
\end{tikzpicture}
\caption{\label{fig:triangle} \small Left: $T_3$ with $10$ elements.  Middle: $T_3+T_3$ with $28$ elements.  Right: $T_3-T_3$ with $37$ elements.}
\end{figure}

Thus, Theorems \ref{thm:1} and \ref{thm:2} do not imply that the proportion of MSTD subsets of $L(nP)$ remains positive in the limit.  In future study, we may begin to examine MSTD sets in higher dimensions by allowing $d$ to depend on $n$---in the case of the triangle set $T_n$, a subset $A\subset T_n$ missing exactly $s$ sums and exactly $2d$ differences is MSTD if and only if $d > s+n^2$.  We discuss this in more detail in Section \ref{sec:conclusion}, and conjecture that the proportion of such subsets approaches $0$ if $L(P)$ is not balanced. At the very least, Theorem \ref{thm:1} implies positive proportions of sum-dominant, balanced, and difference-dominant subsets in the limit if we add the assumption that $L(P)$ is balanced.  It is simple to show that $L(P)$ is balanced if $P$ is point symmetric.  Thus, we have

\begin{corollary} \label{cor:3}
Let $P$ be a convex, point-symmetric polytope in $\mathbb{R}^D$ with vertices in $\mathbb{Z}^D$.  There exists a constant $c>0$ such that, for sufficiently large $n$,
\begin{align}
\#\{ A \subset L(nP) : \text{A is sum-dominant} \}\ > \ c\cdot 2^{|L(nP)|}, \nonumber\\
\#\{ A \subset L(nP) : \text{A is difference-dominant} \} \ >\ c\cdot 2^{|L(nP)|}, \nonumber\\
\#\{ A \subset L(nP) : \text{A is balanced} \}\ >\ c\cdot 2^{|L(nP)|}.
\end{align}
\end{corollary}

\section{Sums and Differences of Edge Elements} \label{sec:edgeElem}

A key idea in past work on MSTD sets is the importance of fringe elements.  For any set $A\subset\{0,\dots,n\}$, there are relatively few ways of forming sums near $0$ and $2n$ and of forming differences near $-n$ and $n$.  Such sums and differences are formed entirely by elements in $A$ near $0$ and $n$---the fringe elements.  On the other hand, there are relatively many ways of forming the respective middle sums and middle differences, and thus they have high probability of being present as we let $A$ vary.  Thus, the sizes of $A+A$ and $A-A$ are predominantly affected by the elements of $A$ in the fringe, and so it is possible to control the balance of sums and differences of $A$ by cleverly fixing those fringe elements.


A similar idea extends to subsets of the lattice points in a polytope.  In this case, the fringe elements are the points near the vertices of the polytope.  In our chosen fixing of the fringe, elements along certain edges, or 1-faces, of the polytope play a particularly important role in controlling the number of missing sums and differences.  To that end, we establish in this section some ancillary lemmas that highlight the behavior of sums and differences of edge elements.

Let $P$ denote our given convex polytope in $\R^D$ with vertices in $\Z^D$.  We begin with the observation that because $P$ has lattice points as its vertices, the dilated polytope $nP$ has at least $n+1$ lattice points along each edge.  More specifically, if an edge $E$ of $P$ contains $b_E+1$ lattice points (where $b_E\ge 1$ since $E$ contains at least its two endpoints), then its dilated form $nE$ in $nP$ contains $nb_E+1$ lattice points.  Furthermore, these $nb_E+1$ lattice points are evenly spaced along the edge, and thus form their own one-dimensional lattice structure.  If $nE$ has endpoints $n\mathbf{e}_1$ and $n\mathbf{e}_2$, then we can define an injective affine transformation $T_{nE}:\R\to\R^D$ by setting
\begin{align}
T_{nE}(x)\ =\ (n\mathbf{e}_2-n\mathbf{e}_1)/(nb_E)\cdot x + n\mathbf{e}_1 \ = \ (\mathbf{e}_2-\mathbf{e}_1)/b_E\cdot x + n\mathbf{e}_1
\end{align}
for all $x\in\R$.  Note $T_{nE}$ forms a one-to-one correspondence between $[0,nb_E]$ and $L(nE)$.  Thus, when constructing a set $A\subset L(nP)$, we can `place' an arbitrarily large, one-dimensional set $S\subset[0,nb_E]$ along any edge $nE$ by taking $n$ to be sufficiently large and then setting $A\cap nE = T_{nE}(S)$.

Lemmas \ref{lemma:edgeSumSet} and \ref{lemma:edgeDiffSet} essentially state that for whatever one-dimensional sets are placed along edges of $nP$, we can find corresponding sumsets along edges in $nP+nP$ and, sometimes, corresponding difference sets along edges in $nP-nP$.

\begin{lemma} \label{lemma:edgeSumSet}
Let $Q$ be a convex polytope in $\mathbb{R}^D$ with vertices in $\mathbb{Z}^D$, let $E$ be an edge of $Q$, and let $A\subset L(Q)$.  Suppose $A\cap E = T_E(S)$, where $S\subset\Z$ and $T_E:\mathbb{R}\to\mathbb{R}^D$ is an injective affine transformation.  Then there exists an injective affine transformation $T_{E+E}:\mathbb{R}\to\mathbb{R}^D$ such that
\begin{align}
(A+A)\cap (E+E) = T_{E+E}(S+S).
\end{align}
\end{lemma}

\begin{proof}
We first show that
\begin{align} \label{eq:edgeSums}
(A+A) \cap (E+E) = (A\cap E) + (A \cap E).
\end{align}
As $(A\cap E) + (A\cap E) \subset (A+A) \cap (E+E)$ is immediate, we need only show the forward inclusion.

Let $\mathbf{k}$ be a point in $E+E$.  By the convexity of $E$, there exists some $\mathbf{e}\in E$ such that $2\mathbf{e} = \mathbf{k}$.  Thus, for any pair of points $\mathbf{a}_1, \mathbf{a}_2 \in A$ with $\mathbf{a}_1 + \mathbf{a}_2 = \mathbf{k}$, we have that $(\mathbf{a}_1 + \mathbf{a}_2)/2 = \mathbf{e}$.  In other words, $\mathbf{a}_1$, $\mathbf{a}_2$ and $\mathbf{e}$ are collinear with $\mathbf{e}$ halfway between $\mathbf{a}_1$ and $\mathbf{a}_2$.  Let $H$ be a supporting hyperplane of $Q$ such that $H\cap Q = E$.  Suppose $\mathbf{a}_1,\mathbf{a}_2\not\in H$.  Since $\mathbf{e}\in E\subset H$, it must be that $\mathbf{a}_1$ and $\mathbf{a}_2$ are in different open half-spaces formed by $H$.  But, since $H$ supports $Q$, then either $\mathbf{a}_1$ or $\mathbf{a}_2$ is not in $Q$---a contradiction.  Thus we have that $\mathbf{a}_1,\mathbf{a}_2\in H$, and therefore $\mathbf{a}_1,\mathbf{a}_2\in E$.  In other words, $(A+A)\cap (E+E)\ \subset\ (A\cap E)+(A\cap E)$, and \eqref{eq:edgeSums} follows.

We now prove the lemma.  We can write $T_E(x) = M(x) + \mathbf{b}$ for all $x\in\R$, where $M:\R\to\R^D$ is an injective linear transformation and $\mathbf{b}\in\R^D$ is some translation vector.  Define $T_{2E}:\R\to\R^D$ such that $T_{2E}(x) = M(x) + 2\mathbf{b}$ for all $x\in\mathbb{R}$.  Since $M$ is injective and linear, $T_{2E}$ is injective and affine.  By \eqref{eq:edgeSums},
\begin{align}
(A+A)\cap 2E\ &=\ (A\cap E) + (A\cap E) \nonumber \\
              &=\ T_E(S) + T_E(S) \nonumber \\
              &=\ (M(S)+\mathbf{b})+(M(S)+\mathbf{b}) \nonumber \\
              &=\ M(S+S) + 2\mathbf{b} \nonumber \\
              &=\ T_{2E}(S+S),
\end{align}
as desired.
\end{proof}

\begin{lemma} \label{lemma:edgeDiffSet}
Let $Q$ be a locally point symmetric polytope in $\mathbb{R}^D$ with vertices in $\mathbb{Z}^D$, and let $A\subset L(Q)$.  For a pair of strictly antipodal vertices $\mathbf{v}_1$ and $\mathbf{v}_2$, let $E_1$ and $E_2$ be parallel edges such that $\mathbf{v}_1 \in E_1$ and $\mathbf{v}_2 \in E_2$.  Suppose $A\cap E_1 = T_{E_1}(S_1)$ and $A\cap E_2 = T_{E_2}(S_2)$, where $S_1,S_2\subset\mathbb{Z}$ and $T_{E_1},T_{E_2}:\mathbb{R}\to\mathbb{R}^D$ are injective affine transformations with the same associated linear transformation.  Then there exists an injective affine transformation $T_{E_2-E_1}:\mathbb{R}\to\mathbb{R}^D$ with
\begin{align}
(A-A)\cap (E_2-E_1)\ =\ T_{E_2-E_1}(S_2-S_1).
\end{align}
\end{lemma}

\begin{proof}
The proof proceeds similarly to that of Lemma \ref{lemma:edgeSumSet}.  We begin by showing that
\begin{align} \label{eq:edgeDiffs}
(A-A) \cap (E_2 - E_1) = (A\cap E_2) - (A\cap E_1).
\end{align}
That $(A\cap E_2) - (A\cap E_2) \subset (A-A)\cap (E_2-E_1)$ is immediate, so we need only show the forward inclusion.  Let $\mathbf{e_1}\in E_1$, $\mathbf{e}_2 \in E_2$.  It suffices to show that if $\mathbf{t}\in\mathbb{R}^D$ and $\mathbf{e}_1+\mathbf{t},\mathbf{e}_2+\mathbf{t}\in Q$, then $\mathbf{e}_1+\mathbf{t}\in E_1$ and $\mathbf{e}_2+\mathbf{t}\in E_2$.

We first show that there exists a pair of parallel supporting hyperplanes $H_1$ and $H_2$ of $Q$ such that $H_1 \cap Q = E_1$ and $H_2 \cap Q = E_2$.  Let $H_1$ be a supporting hyperplane of $Q$ such that $H_1 \cap Q = E_1$, and let $H_2$ be the parallel hyperplane that contains $E_2$.  Suppose there exists some point $\mathbf{q} \in (H_2 \cap Q) \setminus E_2$.  By the convexity of $Q$, we then have that the line segment $\overline{\mathbf{qv}_2}$ is also contained in $H_2$.  Since $\overline{\mathbf{qv}_2}$ cannot be parallel to $E_1$, we have by the local point symmetry of $Q$ that $\overline{\mathbf{qv}_2}$ cannot be an edge of $Q$---otherwise, $H_1 \cap Q$ should contain another edge besides $E_1$ that contains $\mathbf{v}_1$ and is parallel to $\overline{\mathbf{qv}_2}$.  It is not hard to show then that there is some edge of $Q$ other than $E_2$ that is contained in the half-space of $H_2$ that does not contain $E_1$.  By the local point symmetry of $Q$, there must be some corresponding parallel edge of $Q$ other than $E_1$ that is contained in the half-space of $H_1$ that does not contain $E_2$.  As this is not the case, we have that $H_2 \cap Q = E_2$.

Now let $V_1$ denote the closed half-space formed by $H_1$ that contains $Q$, and $V_2$ the closed half-space formed by $H_2$ that contains $Q$.  Note that if a translation vector $\mathbf{t}\in\mathbb{R}^D$ does not lie in $H_1$ (or $H_2$), then either $\mathbf{e}_1+\mathbf{t}\not\in V_1$ or $\mathbf{e}_2+\mathbf{t}\not\in V_2$.  Thus if $\mathbf{e}_1+\mathbf{t},\mathbf{e}_2+\mathbf{t}\in Q$, then $\mathbf{t}\in\mathbb{R}^D$ must lie in $H_1$.  Then $\mathbf{e}_1+\mathbf{t}\in H_1$ and $\mathbf{e}_2+\mathbf{t}\in H_2$.  Since $H_1\cap Q = E_1$ and $H_2\cap Q = E_2$, it follows that $\mathbf{e}_1+\mathbf{t}\in E_1$ and $\mathbf{e}_2+\mathbf{t}\in E_2$, and thus \eqref{eq:edgeDiffs} follows.

We now prove the lemma.  We can write $T_{E_1}(x) = M(x)+\mathbf{b}_1$ and $T_{E_2}(x) = M(x) + \mathbf{b}_2$ for all $x\in\R$, where $M:\R\to\R^D$ is an injective linear transformation and $\mathbf{b}_1,\mathbf{b}_2\in\R^D$ are translation vectors.  Define $T_{E_2-E_1}:\R\to\R^D$ such that $T_{E_2-E_1}(x) = M(x) + (\mathbf{b}_2-\mathbf{b}_1)$ for all $x\in\R$.  Since $M$ is injective and linear, $T_{E_2-E_1}$ is injective and affine.  By \eqref{eq:edgeDiffs},
\begin{align}
(A-A)\cap (E_2-E_1)\ &=\ (A\cap E_2) - (A\cap E_1) \nonumber \\
                     &=\ T_{E_2}(S_2) - T_{E_1}(S_1) \nonumber \\
                     &=\ (M(S_2) + \mathbf{b}_2) - (M(S_1) + \mathbf{b}_1) \nonumber \\
                     &=\ M(S_2-S_1) + (\mathbf{b}_2-\mathbf{b}_1) \nonumber \\
                     &=\ T_{E_2-E_1}(S_2-S_1),
\end{align}
as desired.
\end{proof}

\begin{definition}
Given a set $S\in\R^D$, a difference vector $\mathbf{k}\in S-S$ is \emph{uniquely formed} if there exists a unique pair of elements $\mathbf{s}_1, \mathbf{s}_2 \in S$ satisfying $\mathbf{s}_1 - \mathbf{s}_2 = \mathbf{k}$.
\end{definition}

The remainder of this section is devoted to proving Lemma \ref{lemma:uniqueDiff}, which asserts that there are many (at least on the order of $n$) uniquely formed differences in $nP-nP$ if $P$ is not locally point symmetric.  By contrast, if $P$ is locally point symmetric, then the number of uniquely formed differences in $nP-nP$ is constant, as we will show in Lemma \ref{lemma:verticesofdiffset}.

Showing Lemma \ref{lemma:uniqueDiff} requires a brief review of geometry.  In the following definitions, let $Q$ be a convex polytope in $\R^D$.  Further assume that $Q$ is $D$-dimensional---that is, the smallest affine subspace containing $Q$ is $\R^D$.

\begin{definition}
Given vectors $\mathbf{x}_1, \mathbf{x}_2, \dots, \mathbf{x}_m \in \R^D$, a \emph{conical combination} of these vectors is a vector of the form $\alpha_1\mathbf{x}_1 + \alpha_2\mathbf{x}_2 + \dots + \alpha_m\mathbf{x}_m$ where $\alpha_i \ge 0$ for all $1\le i\le m$.  The \emph{polyhedral cone} generated by vectors $\mathbf{x}_1, \mathbf{x}_2, \dots, \mathbf{x}_m$ is the set of all conical combinations of $\mathbf{x}_1, \mathbf{x}_2, \dots, \mathbf{x}_m$.  
\end{definition}

\begin{definition}
Let $\mathbf{v}$ be a vertex of $Q$, and let $\mathbf{n}_1, \dots , \mathbf{n}_t$ denote outward-pointing normal vectors of all facets of $Q$ that contain $\mathbf{v}$.  The \emph{normal cone} $N(\mathbf{v})$ of $Q$ at $\mathbf{v}$ is the polyhedral cone generated by $\mathbf{n}_1, \dots , \mathbf{n}_t$.
\end{definition}

Note that normal cones have their apexes at the origin of $\mathbb{R}^d$, while supporting cones have their apexes at the vertices of the polytope.

Suppose $Q$ has vertices $\mathbf{v}_1, \mathbf{v}_2, \dots, \mathbf{v}_m$.  The following properties of normal cones $N(\mathbf{v}_i)$ are easily verified:
\begin{enumerate}
\item For each vertex $\mathbf{v}_i$, the normal cone $N(\mathbf{v}_i)$ is the set of outward normal vectors (of arbitrary length) to all supporting hyperplanes of $Q$ that contain $\mathbf{v}_i$.
\item Vertices $\mathbf{v}_i$ and $\mathbf{v}_j$ are strictly antipodal if and only if the interiors of $N(\mathbf{v}_i)$ and $-N(\mathbf{v}_j)$ have non-empty intersection.
\item For $i\not=j$, the interiors of $N(\mathbf{v}_i)$ and $N(\mathbf{v}_j)$ are disjoint.
\item For $i\not=j$, the intersection of $N(\mathbf{v}_i)$ and $N(\mathbf{v}_j)$ is either $\{\mathbf{0}\}$ or a facet of both cones.
\item $\bigcup_{i=1}^m N(\mathbf{v}_i) = \R^D$.
\end{enumerate}

We now introduce a useful result that follows easily from the work of Nguy\^{e}n and Soltan \cite{NS}.  We provide the details in Appendix \ref{appendix:A}.


\begin{lemma} \label{lemma:saVertices}
Let $Q$ be a $D$-dimensional polytope with $m$ vertices in $\R^D$.  Then $Q$ is locally point symmetric if and only if $Q$ has exactly $m/2$ pairs of strictly antipodal vertices.
\end{lemma}

We are now ready to prove Lemma \ref{lemma:uniqueDiff}.

\begin{lemma} \label{lemma:uniqueDiff}
Let $Q$ be a convex polytope in $\mathbb{R}^D$ that is not locally point symmetric.  Then there is a vertex $\mathbf{v}$ and an edge $E$ of $Q$ such that, for all $\mathbf{e}\in E$, the difference vectors $\kk = \mathbf{e}-\mathbf{v}$ and $-\kk$ are uniquely formed.
\end{lemma}

\begin{proof}
The difference vectors $\kk = \mathbf{e}-\mathbf{v}$ and $-\kk$ are uniquely formed if and only if there exists no non-zero vector $\mathbf{t}\in\R^D$ such that $\mathbf{e}+\mathbf{t}, \mathbf{v}+\mathbf{t} \in Q$.  To show that such a vertex $\mathbf{v}$ and an edge $E$ exist, it suffices to show that there exists a pair of parallel supporting hyperplanes $H_1$ and $H_2$ of $Q$ such that $H_1 \cap Q = \{\mathbf{v}\}$ and $H_2 \cap Q = E$.  This is clear because, for any translation by a vector $\mathbf{t}\in\R^D$ of $\mathbf{v}$ and some $\mathbf{e}\in E$, we must have that $\mathbf{t}$ is parallel to $H_1$ and $H_2$ if $\mathbf{e}+\mathbf{t}$ and $\mathbf{v}+\mathbf{t}$ are to remain in the closed space bounded by $H_1$ and $H_2$.  But $H_1\cap Q = \{\mathbf{v}\}$, and therefore it must be that $\mathbf{t} = \mathbf{0}$ if $\mathbf{v}+\mathbf{t}\in Q$.  See Figure \ref{fig:uniqueDiff} for an illustration.

\begin{figure}[b]
\centering
\begin{tikzpicture}[line cap=round,line join=round,>=triangle 45,x=1.0cm,y=1.0cm]
\clip(-0.42,-0.46) rectangle (8.18,6.46);
\draw (0.56,1.98)-- (4.94,5.62)-- (6.88,3.4)-- (5.04,1.12)-- (0.56,1.98);
\draw [domain=-0.42:8.18] plot(\x,{(--21.87-2.22*\x)/1.94});
\draw [domain=-0.42:8.18] plot(\x,{(--5.08-2.22*\x)/1.94});
\draw [->] (0.56,1.98) -- (5.59,4.88);
\begin{scriptsize}
\fill[color=black] (0.56,1.98) circle (1.5pt);
\draw[color=black] (0.3,1.72) node {$\mathbf{v}$};
\draw[color=black] (6.4,4.5) node {$E$};
\fill [color=black] (5.59,4.88) circle (1.5pt);
\draw[color=black] (5.85,5.1) node {$\mathbf{e}$};
\draw[color=black] (3.9,3.5) node {$\mathbf{k}$};
\draw (1.8,0) node {$H_1$};
\draw (7.4,2.3) node {$H_2$};
\end{scriptsize}
\end{tikzpicture}
\caption{\small \label{fig:uniqueDiff} A quadrilateral $Q$ that is not locally point symmetric.  The parallel lines $H_1$ and $H_2$ support $Q$ precisely at $\mathbf{v}$ and at $E$, respectively.  For any non-zero translation vector $\mathbf{t}$, either $\mathbf{v}+\mathbf{t}\not\in Q$ or $\mathbf{e}+\mathbf{t}\not\in Q$, and the difference vector $\mathbf{k} = \mathbf{e}-\mathbf{v}$ is uniquely formed.}
\end{figure}
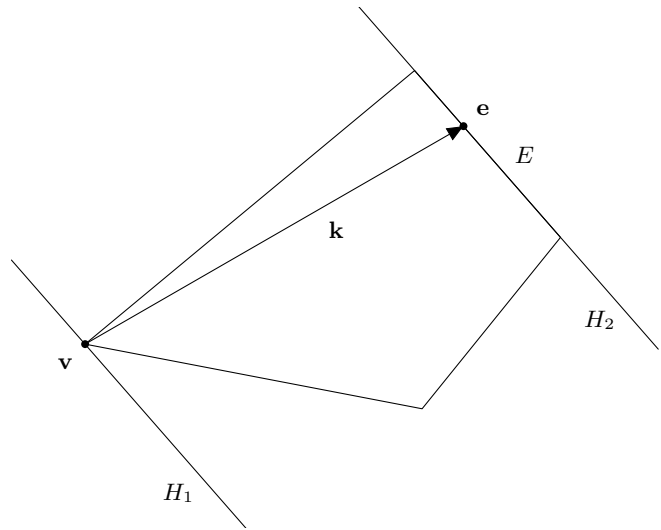

First assume that $Q$ is $D$-dimensional, with $m$ vertices.  As $Q$ is not locally point symmetric, and every vertex of a convex polytope is strictly antipodal with at least one other vertex
, it follows by Lemma \ref{lemma:saVertices} that the number of pairs of strictly antipodal vertices is strictly greater than $m/2$.  Then there exists some vertex $\mathbf{v}$ of $Q$ that is strictly antipodal with at least two other vertices.  Let $\mathbf{u}_1$ and $\mathbf{u}_2$ denote two such vertices.  By property (2) above, the interiors of $N(\mathbf{v})$ and $-N(\mathbf{u}_1)$ have non-empty intersection, as do the interiors of $N(\mathbf{v})$ and $-N(\mathbf{u}_2)$.  For the sake of contradiction, suppose that $N(\mathbf{v})$ is contained in $-N(\mathbf{u}_1)$.  Reflection through the origin is injective, and the interiors of $N(\mathbf{u}_1)$ and $N(\mathbf{u}_2)$ are disjoint by property (3) above, so it follows that the interiors of $N(\mathbf{v})$ and $-N(\mathbf{u}_2)$ are disjoint---a contradiction.  Thus, $N(\mathbf{v})$ cannot be contained in $-N(\mathbf{u}_1)$.

As the interiors of $N(\mathbf{v})$ and $-N(\mathbf{u}_1)$ still have non-empty intersection, it is not hard to show that the interior of some facet $F$ of $-N(\mathbf{u}_1)$ has non-empty intersection with the interior of $N(\mathbf{v})$.  Now note that $F$ is also a facet of the cone $-N(\mathbf{u}')$ for some vertex $\mathbf{u}'$ that is connected to $\mathbf{u}_1$ by an edge---we let $E$ denote this edge.  Further note that $F$ is set of (inward-pointing) normal vectors of supporting hyperplanes $H$ of $Q$ that satisfy $H\cap Q = E$.  In other words, there exist parallel supporting hyperplanes $H_1$ and $H_2$ of $Q$ such that $H_1 \cap Q = \{\mathbf{v}\}$ and $H_2\cap Q = E$, as desired.


If $Q$ is not $D$-dimensional---that is, the dimension of the affine hull of $Q$ is some $D'<D$---then we can define some injective affine transformation $T:\text{aff}(Q) \to \R^{D'}$ from the affine hull of $Q$ to $\R^{D'}$.  As affine transformations preserve parallel lines, the image polytope $Q' = T(Q)$ is also not locally point symmetric.  It is not hard to show that a difference vector $\mathbf{q}'_1 - \mathbf{q}'_2 \in Q'-Q'$ is uniquely formed if and only if $T^{-1}(\mathbf{q}'_1) - T^{-1}(\mathbf{q}'_2) \in Q - Q$ is uniquely formed.  Thus, we can prove the lemma for $Q$ by applying the argument above to $Q'$.
\end{proof}

\section{Middle Sums and Differences}

Let $k < n/2$, let $A \subset \{0,\dots,n\}$, and define sets $L := A\cap [0,k]$ and $U:= A\cap [n-k,n]$.  It is easy to see that
\begin{align}
(A+A)\cap([0,k]\cup[2n-k,2n])\ &\subset\ (L+L)\cup(U+U), \nonumber\\
(A-A)\cap([-n,-n+k]\cup[n-k,n])\ &\subset\ (L-U)\cup (U-L).
\end{align}
In other words, the sums and differences within radius $k$ of the endpoints of the potential sumset and potential difference set, respectively, are formed entirely by the fringe elements within radius $k$ of the endpoints of the base set $\{0,\dots,n\}$.  Martin and O'Bryant exploit this idea in \cite{MO} by fixing the fringe of $A$ such that $A+A$ necessarily has more elements at its ends than does $A-A$.  They then show that with high probability, all `middle' sums and differences are present in the sumset and difference set.

The same idea extends to higher-dimensional convex polytopes.  Given an arbitrary convex polytope $Q$ and some $r>0$, define sets
\begin{align}
B_r(Q) &\ := \  \{ \mathbf{q}\in L(Q) : \text{$d(\mathbf{q},\mathbf{v}) \le r$ for some vertex $\mathbf{v}$ of $Q$}\}, \nonumber\\
M_r(Q) &\ := \  L(Q) \setminus B_r(Q),
\end{align}
where $d(\cdot ,\cdot )$ denotes the Euclidean metric.  In words, $B_r(Q)$ is the set of lattice points contained in the union of balls of radius $r$ centered at the vertices of $Q$, while $M_r(Q)$ consists of all other `middle' lattice points.  It is easy to show that for any $A\subset L(nP)$,
\begin{align}
(A+A) \cap B_r(nP+nP)\ &\subset\ (A \cap B_r(nP)) + (A \cap B_r(nP)), \nonumber\\
(A-A) \cap B_r(nP-nP)\ &\subset\ (A \cap B_r(nP)) - (A \cap B_r(nP)).
\end{align}
Thus, we can precisely control the fringe of the sumset and difference set---$(A+A)\cap B_r(nP+nP)$ and $(A-A)\cap B_r(nP-nP)$---by carefully fixing $A \cap B_r(nP)$, the fringe of $A$.  Importantly, we choose $r$ independently of the dilation factor $n$, fixing a constant number of points as $n$ grows.

We refer to any other possible sum---that is, an element of $M_r(nP+nP)$---as a \emph{middle sum}, and any other possible difference---that is, an element of $M_r(nP-nP)$---as a \emph{middle difference}.  If we can show that all middle sums and all middle differences are present with positive probability, then we have a positive proportion of subsets $A\subset L(nP)$ that satisfy some precise condition on the cardinalities of their sumsets and difference sets.  The purpose of this section is to show that this is true if the fringe is large enough and, in the case of middle differences, if and only if $P$ is locally point symmetric.

\begin{proposition} \label{prop:middleSums}
Let $0<p^+<1$ be given.  Then there exists some $r>0$ such that for all sufficiently large $n$, the following holds:\ \ Let $F_r\subset B_r(nP)$, and let $A$ be uniformly randomly chosen from all subsets $S\subset L(nP)$ such that $S\cap B_r(nP) = F_r$.  Then $M_r(nP+nP) \subset A+A$ with probability at least $p^+$.
\end{proposition}

\begin{proof}
We begin with a lemma bounding the probability that any individual middle sum is missing.

\begin{lemma} \label{lemma:probMissingSum}
Let $r>0$, and fix a fringe set $F_r \subset B_r(nP)$.   Let $A$ be chosen uniformly at random from all subsets $S\subset L(nP)$ such that $S\cap B_r(nP) = F_r$, and let $\kk\in M_r(nP+nP)$.  Then, for some constant $c > 0$ independent of $n$,
\begin{equation}
\mathbb{P}[\mathbf{k} \not\in A+A]\ \le\ c \left( \frac34 \right)^{ \left| L (nP \cap (\kk - nP)) \right|/2}.\end{equation}
\end{lemma}

\begin{proof}
The proof is similar to that of Lemma 5 in Martin and O'Bryant \cite{MO}.  Suppose we have $\x, \y \in nP$ such that $\x+\y=\kk$.  Then $\x = \kk - \y \in \kk - nP$, and similarly $\y \in \kk-nP$.  Then $L(nP \cap (\kk - nP))$ can be partitioned into distinct pairs of lattice points that add up to $\kk$, and the singleton $\{\kk/2\}$ if $\kk/2$ is a lattice point.  The probability that $\kk$ is missing in $A+A$ is then the product of the independent probabilities that in each pair, at least one point is missing.  Suppose that in our fixed fringe set $F_r$, exactly $l$ points are fixed as missing.  Then at most $l$ pairs contribute a probability of $1$, and the remaining pairs contribute a probability of at most $3/4$.  When $\kk/2$ is not a lattice point, there are $\left| L (nP \cap (\kk - nP)) \right|/2$ pairs total, which gives
\begin{align}
\mathbb{P}[\mathbf{k} \not\in A+A]\ \le\  \left( \frac34 \right)^{ \left| L (nP \cap (\kk - nP)) \right|/2 - l}.
\end{align}
Thus, we may take $c = (3/4)^{-l}$ and the lemma follows.  In the case where $\kk/2$ is a lattice point, a similar argument gives the same bound. \end{proof}

By the union bound, the probability that at least one middle sum is missing is at most the sum of the probabilities that each individual middle sum is missing.  Thus, to prove Proposition \ref{prop:middleSums}, it suffices to show
\begin{align} \label{eq:sumMissingSums}
\sum_{\kk \in M_r(nP+nP)} c \left( \frac34 \right)^{ \left| L (nP \cap (\kk - nP)) \right|/2}\ <\ 1 - p^+
\end{align}
for sufficiently large $n$ and $r$.

In the one-dimensional case, this amounts to making a tail of a geometric series as small as desired, which is done in \cite{MO}.  Unfortunately, in $D$ dimensions, the shape $nP \cap (\mathbf{k} - nP)$ can get quite complicated, and we must do more work.  The key idea is that when $\mathbf{k}$ is close to a vertex of $nP+nP$, the shape $nP \cap (\mathbf{k} - nP)$ is a parallelotope, which is quite simple.  Conversely, when $\mathbf{k}$ is not close to a vertex of $nP+nP$, we are saved by the fact that $nP \cap (\mathbf{k} - nP)$ is large, so we do not need to be careful about counting its lattice points.  See Figure \ref{fig:sumsetIntersection} for an illustration.

\begin{figure}
\begin{tikzpicture}[line cap=round,line join=round,>=triangle 45,x=0.55cm,y=.55cm]

\draw [step=0.5,thin,gray!40] (-1.4,-1.4) grid (10.9,8.9);
\draw [->] (-1.5,0) -- (11,0);
\draw [->] (0,-1.5) -- (0,9);

\draw 
(0.0,0.0)--
(5.0,0.0)-- 
(1.0,4.0)-- 
(0.0,0.0);

\draw (-0.5,-0.5) node {\tiny $A$};
\draw (5.4,-0.5) node {\tiny $B$};
\draw (0.6,4.2) node {\tiny $C$};

\draw 
(0.0,0.0)-- 
(2.0,8.0)-- 
(10.0,0.0)-- 
(0.0,0.0);

\draw (10.4,-0.5) node {\tiny $D$};
\draw (2.0,8.5) node {\tiny $E$};

\draw 
(4.0,3.0)-- 
(-1.0,3.0)-- 
(3.0,-1.0)-- 
(4.0,3.0);

\draw (4.4,3.4) node {\tiny $\kk$};
\fill (4.0,3.0) circle (1.5pt);

\draw [step=0.5,thin,gray!40] (16-2.9,-1.4) grid (16+10.9,8.9);
\draw [->] (13,0) -- (27,0);
\draw [->] (16,-1.5) -- (16,9);

\draw 
(16.0,0.0)-- 
(18.0,8.0)-- 
(26.0,0.0)-- 
(16.0,0.0);

\draw (16+10.4,-0.5) node {\tiny $D$};
\draw (16+2.0,8.5) node {\tiny $E$};

\draw 
(16.0,0.0)-- 
(17.0,4.0)-- 
(21.0,0.0)-- 
(16.0,0.0);

\draw (16-0.5,-0.5) node {\tiny $A$};
\draw (16+5.4,-0.5) node {\tiny $B$};
\draw (16+0.6,4.2) node {\tiny $C$};

\draw 
(18.5,7.0)-- 
(17.5,3.0)-- 
(13.5,7.0)-- 
(18.5,7.0);

\draw (18.9,7.4) node {\tiny $\kk$};
\fill (18.5,7.0) circle (1.5pt);

\end{tikzpicture}
\caption {\small \label{fig:sumsetIntersection} A triangle $T = \triangle ABC$ with its sumset $T+T = \triangle ADE$.  On the left, the sum $\kk$ is relatively far from the vertices of $T$.  Thus, while the shape of the intersection $T\cap (\kk-T)$ is hard to control, it fortunately contains many lattice points.  On the right, the sum $\kk$ is relatively close to a vertex of $T$, and hence $T\cap (\kk-T)$ is a parallelotope.}
\end{figure}
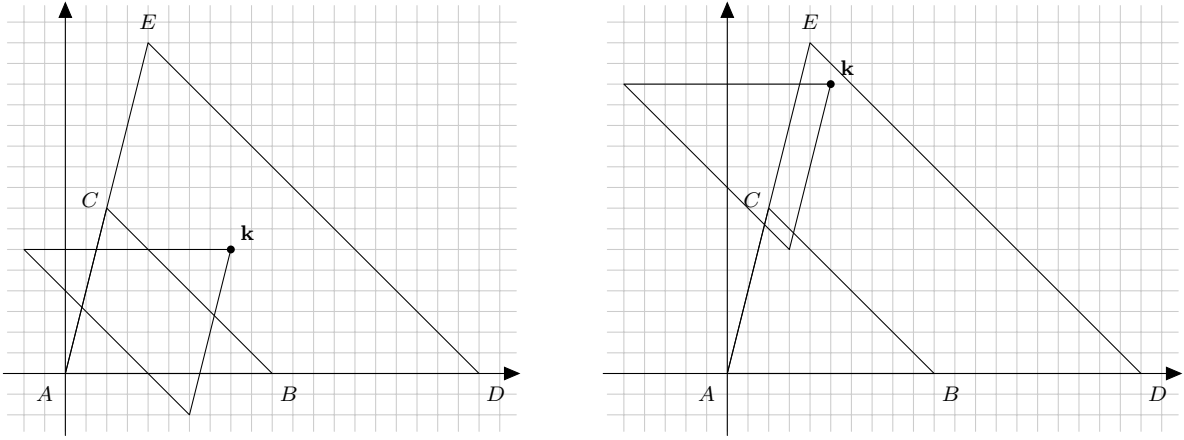

To evaluate the sum in \eqref{eq:sumMissingSums}, we partition $M_r(nP+nP)$ into two sets, a `center' set $C$ and an `intermediate' set $I$.  Fix some $\varepsilon_C > 0$.  We define $C$ as all points $\kk\in M_r(nP+nP)$ such that
\begin{align}
|L(nP\cap (\kk-nP))|\ >\ 2 \log_{3/4} \left(\frac{\varepsilon_C / c}{L(nP+nP)}\right).
\end{align}
The right hand side is constructed so that by Lemma \ref{lemma:probMissingSum},
\begin{align} \label{eq:Cbound}
\mathbb{P}[\kk \not\in A+A] \ <\ \frac{\varepsilon_C}{ |L(nP+nP)|}
\end{align}
for all $\kk\in C$.  We conclude that the sum of $\mathbb{P}[\kk \not\in A+A]$ over all $\kk$ in $C$ is at most $\varepsilon_C$, because $C \subseteq L(nP+nP)$.

We define $I$ to consist of the remaining points; that is, all points $\kk\in M_r(nP+nP)$ such that
\begin{align}
|L(nP\cap (\kk-nP))|\ \le\ 2 \log_{3/4} \left(\frac{\varepsilon_C / c}{|L(nP+nP)|}\right). \label{logn}
\end{align}
Note that the right hand side is $\Theta(\log n)$, while $|L(nP+nP)|$ is $\Theta(n^D)$.  Intuitively, the above definition suggests that all points $\kk \in I$ should lie close to the vertices of $nP+nP$, in order to make $|L(nP \cap (\kk - nP)|$ small.  The set $I$ will be located in between the fringe set $F_r$ and the center set $C$.  Since $\log n / n \to 0$ as $n \to \infty$, the intersection polytope $nP \cap (\mathbf{k} - nP)$ will be a parallelotope for all $\mathbf{k} \in I$, for $n$ large.  We formalize this idea in the following lemma.

\begin{lemma} \label{lemma:tlogballs}
There exists a constant $t > 0$ such that $I$ is contained in the union of balls of radius $t \log n$ around the vertices of $nP+nP$ for all $n$.
\end{lemma}

\begin{proof} To highlight the dependence of $I$ on $n$, we write $I(n)$.  Assume for the sake of contradiction this lemma is false.  Then for each $t$, there is some $n$ such that $I(n)$ is not contained in the union of balls of radius $t \log n$ centered at the vertices of $2nP$.  In particular, letting $t$ take on the value of every positive integer $m$, we have the following: For each positive integer $m$ there exists $\mathbf{k}_m$ and $n_m$ such that $\mathbf{k}_m \in I(n_m) $ but $\mathbf{k}_m$ has distance greater than $m \log n$ from each vertex of $2nP$.

For the next step, it is useful to visualize the polytope $P$ as fixed, and rather than dilating $P$ by a factor of $n$, we shrink the underlying lattice by a factor of $n$.

Consider the sequence $\mathbf{k}'_m = \mathbf{k}_m/m$.  Note that this sequence lies inside the polytope $P+P$, which is closed and bounded, so there is a convergent subsequence $\mathbf{k}'_{m_i}$.  Let $\mathbf{k}'$ denote the limit of this subsequence.\\ \

\textbf{Step 1.}  We claim that $\mathbf{k}'$ must be a vertex of $P+P$; that is, $P \cap (\kk' - P)$ is just one point (the point $\kk'/2$, a vertex of both $P$ and $\kk' - P$).

Consider the convex polytope $P \cap (\kk' - P)$, and suppose that it is not just a point.  Then it is a $d'$-dimensional polytope where $0 < d' \le D$, and furthermore it must lie in a $d'$-face of $P$ (and of $\kk' - P$).  Since $P$ has lattice point vertices, a $d'$-face of $P$ contains $\Theta(n^{d'})$ lattice points.  (Recall that $P$ stays fixed and the lattice shrinks.)

Let $A$ be the affine $d'$-dimensional subspace containing $P \cap (\kk' - P)$.  Relative to $A$, the point $\kk'/2$ is in the interior of $P \cap (\kk' - P),$ with some positive distance $\epsilon$ to all its bounding faces.  So let $B \subset A$ be a $d'$-dimensional ball centered at $\kk'/2$ with fixed radius $\epsilon/2$, so that $B \subset P \cap (\kk' - P)$.  Then $B$ contains $\Theta(n^{d'})$ lattice points.

We now show that $B \subset P \cap (\kk'_{m_i} - P)$ for $i$ large.   The fact that $\kk'_{m_i} \to \kk'$ as $i \to \infty$ is sufficient for this.  First, consider any hyperplane $H$ making an angle $\theta$ with $A$, and suppose $H$ is translated normally by a distance at most $\epsilon$.  Then the intersection $H \cap A$ is translated along $A$ by a distance at most $\epsilon / \sin \theta$. In particular, among all bounding hyperplanes of $\kk' - P$, there is a minimum nonzero angle $\theta_{{\rm min}}$ made with $A$, so that whenever $\kk'$ is translated by at most $\delta = \epsilon/(2 \sin \theta_{{\rm min}})$, the bounding faces of $P \cap (\kk' - P) \cap A$ are translated by at most $\epsilon / 2$.  Finally, for $i$ large, $|\kk'_{m_i} - \kk'| < \delta$, so $B \subset P \cap (\kk'_{m_i} - P)$, as desired.  Since $B$ contains $\Theta(n^{d'})$ lattice points, this directly contradicts that $\kk_{m_i} \in I(n)$, which says that $|L(P \cap (\kk'_{m_i} - P))| = O(\log n)$.  Thus $\kk'$ must in fact be a vertex of $P+P$, so that $P \cap (\kk' - P) = \{ \kk'/2 \}$. \\ \

\textbf{Step 2.}  Recall that $\kk'/2$ is a vertex of $P$.  For $i$ large, $\kk_{m_i}$ is so close to $\kk$ that $P \cap (\kk_{m_i}  - P) = C(\kk'/2) \cap (\kk'_{m_i} - C(\kk'/2)$, where $C(\kk'/2)$ denotes the supporting cone of $P$ at $\kk'/2$.  In other words, for $i$ large, only the local shape of $P$ at $\kk'/2$ matters: the only hyperplanes determining $P \cap (\kk'_{m_i}  - P)$ are those of $P$ at $\kk'/2$ and the corresponding hyperplanes in $\kk'_{m_i}  - P$.

This means that the shape $P \cap (\kk'_{m_i} - P)$ is quite simple, and the number of its lattice points will be easy to analyze.
Suppose $P \cap (\kk'_{m_i}  - P)$ is a $d'$-dimensional polytope.  Pick $d'$ edges of $P$ at $\kk'$, extend them to rays from $\kk'$, and call their convex hull $P'$.  Then $P' \cap (\kk'_{m_i} - P')$ is a parallelotope---as simple a shape as we could hope for.  Furthermore, that parallelotope is a subset of $P \cap (\kk'_{m_i}  - P)$, so it also has $O(\log n)$ lattice points.  Since each of these $d'$ edges has a lattice structure, the edges have length $O(\log n)$ as well, so the diameter of the parallelotope is $O(\log n)$.  
Thus $|\kk'_{m_i} - \kk'|$ is indeed $O(\log n)$. \end{proof}

Now we evaluate the sum in \eqref{eq:sumMissingSums} over points $\kk\in I$.  We sum around one vertex of $nP+nP$ at a time.  Let $\mathbf v$ be the current vertex, and let $I_\mathbf{v}$ be the portion of $I$ that lies in the ball of radius $t \log n$ about $\mathbf v$.

Since $\log n / n \to 0$ as $n \to \infty$, for $n$ large we have $(nP+nP) \cap B_{\mathbf v} = C(\mathbf v) \cap B_{\mathbf v}$, where $C(\mathbf v)$ denotes the supporting cone of $nP+nP$ at $\mathbf v$.  Henceforth, assume $n$ is this large.  Now the only relevant portion of $nP+nP$ is a neighborhood of $\mathbf v$; that is, when $\mathbf k \in I_{\mathbf v}$
\begin{equation} nP \cap (\mathbf k - nP)\ =\ C(\mathbf{v}) \cap (\mathbf k - C(\mathbf v)).\end{equation}

To show that the sum in \eqref{eq:sumMissingSums}  is small, we show that the sum
\begin{equation}\sum_{\kk \in L(C(\mathbf{v}))} c\left( \frac{3}{4} \right)^{ |L(C(\mathbf v) \cap (\kk - C(\mathbf v) ) )| / 2 }\end{equation}
converges.  Because the terms are positive, it suffices to bound this sum above.  The reason we want to prove convergence is that our final step will be to bound \eqref{eq:sumMissingSums} by an arbitrarily small tail of this sum (recall that we will be cutting out a constant fixed fringe region of radius $r$ around $\mathbf{v}$, and we can make $r$ as large as desired).

Recall that $C(\mathbf v)$ is the convex hull of rays from $\mathbf{v}$ corresponding to edges of $P$.  Then $C(\mathbf v)$ is the union of convex hulls of $D$-tuples of those rays.  Since there are finitely many $D$-tuples, it suffices to show that the sum is bounded in each such region.

Let $R$ be one such region, the convex hull of $D$ rational-slope rays from $v$.  We wish to show that \begin{equation}\sum_{\kk \in L(R)} c\left( \frac{3}{4} \right)^{ |L(C(\mathbf v) \cap (\kk - C(\mathbf v) )) | / 2 }\end{equation} converges.  Since $R \subset C(\mathbf v)$,  \begin{equation}|L(R \cap (\kk - R))|\ <\ |L(C(\mathbf v) \cap (\kk - C(\mathbf v)))|,\end{equation} so it suffices to show that \begin{equation}S_R\ :=\ \sum_{\kk \in L(R)} c\left( \frac{3}{4} \right)^{ |L(R \cap (\kk - R )) | / 2 }\end{equation} converges. This is easier, because $R \cap (\kk - R)$ is simply a parallelotope for any $\kk \in L(R)$.

By induction on $D$ the sum over any facet of $R$ converges, because a facet of $R$ is the convex hull of $D-1$ rays from $v$, and the base case $n=1$ amounts to a geometric series.  Now that the boundary of $R$ has been dealt with, it remains to sum over the lattice points in the interior of $R$, which we shall denote $R^\circ$.  When $\kk$ is in the interior, $R \cap (\kk - R)$ has non-zero volume.  In fact, in the interior, there is a positive constant $c_1$ (depending on $R$) allowing us to bound the number of lattice points below by the volume. That is, for all $\kk \in L(R^\circ)$, \begin{equation} c_1 |R \cap (\kk - R )|\ <\ |L(R \cap (\kk - R ))|.\end{equation}

Using this upper bound, it suffices to show the convergence of
\begin{equation} S'_R\ :=\ \sum_{\kk \in L(R^\circ)} c\left(\frac{3}{4}\right)^{c_1|R \cap (\kk - R ) |}. \end{equation}

We can upper bound the resulting sum further.  Let $T$ be a rational affine transformation that maps $R$ onto the first orthant.  Because $T$ increases the volumes in the exponents by at most a constant factor $c_2$, applying $T$ gives us the new upper bound
\begin{equation} S'_R\ \le\ \sum_{\kk \in T(L(R^\circ))} c\left(\frac{3}{4}\right)^{ c_1 |T(R) \cap (\kk - T(R) ) |/c_2}. \end{equation}

Now, notice that $T(L(R^\circ))$ is a subset of the lattice $T(\mathbb{Z}^d)$ whose points all have positive (rational) coordinates.  Thus, for some rational $q > 0$, we have $L(R^\circ) \subset q \mathbb{N}^D$ and we may further bound the sum above by
\begin{equation} S'_R\ \le\ \sum_{\kk \in q\mathbb{N}^D} c\left(\frac{3}{4}\right)^{ c_1 |T(R) \cap (\kk - T(R) ) |/c_2}. \label{srprime} \end{equation}

Let $x = (3/4)^{c_1/c_2}$.  Since $T(R)$ is equal to the first orthant, $T(R) \cap (\kk - T(R))$ is simply a rectangular cell with opposite vertices $0$ and $\kk$, so our sum in \eqref{srprime} is equal to
\begin{equation} S_R''\ :=\ q^D \sum_{(k_1, k_2, \dots, k_D) \in \mathbb{N}^d} x^{k_1 k_2 \cdots k_D}.\end{equation}

To show that this sum converges, we may rewrite the sum as \begin{equation}S_R''\ =\ q^D \sum_{m \in \mathbb{N}} \psi(m) x^m,\end{equation} where $\psi(m)$ is the number of ways of writing $m$ as the ordered product of $D$ positive integers.  However, $\psi(m)$ is clearly bounded by  $m^D$, so since $x < 1$ the sum converges as desired.

Since the upper bound converges, our original sum \begin{equation} \sum_{\kk \in L(C(\mathbf v))} \left( \frac{3}{4} \right)^{  |L(nP \cap (\kk - nP )) | / 2}\end{equation} also converges.  It follows that by making the constant fixed fringe radius $r$ large enough, we can force the tail sum \begin{equation} \sum_{\kk \in (C(\mathbf{v}) \setminus B_r(nP+nP)) } \left( \frac{3}{4} \right)^{ |L(nP \cap (\kk - nP )) | / 2 }\end{equation} to be smaller than any $\varepsilon_{I_\mathbf{v}}$.
Since $I_\mathbf{v}$ lies in $L(nP_v \setminus B_r(nP+nP))$, the sum over $\kk \in I_\mathbf{v}$ is also smaller than $\varepsilon_{I_\mathbf{v}}$. Thus if we let $\varepsilon_I$ be the sum of $\varepsilon_{I_\mathbf{v}}$ over all vertices $\mathbf{v}$ of $nP+nP$, the probability that at least one middle sum is missing is at most $\varepsilon_M + \varepsilon_I$. In particular, if we choose $\varepsilon_M$ and $\varepsilon_I$ so that $\varepsilon_M + \varepsilon_I < 1 - p^+$, we have at least a constant positive probability $p^+$ that all middle sums are present, as desired.  This concludes the proof of Proposition \ref{prop:middleSums}. \end{proof}

We now examine the presence of middle differences.

\begin{proposition} \label{prop:middleDiffs}
Let $0<p^-<1$.  Suppose $P$ is locally point symmetric.  There exists some $r > 0$ such that for all sufficiently large $n$, the following holds:\ \ Let $F_r\subset B_r(nP)$, and let $A$ be uniformly randomly chosen from all subsets $S\subset L(nP)$ such that $S\cap B_r(nP) = F_r$.  Then $M_r(nP-nP) \subset A-A$ with probability at least $p^-$.
\end{proposition}

\begin{proof} The proof is largely identical to the proof of Proposition \ref{prop:middleSums}.  We highlight the relevant differences here. In the course of the proof we state and prove two useful lemmas.

First of all, when considering sums, the pairs of points in $L(nP)$ that sum up to some $\mathbf{k} \in L(nP) + L(nP)$ are pairwise disjoint.  Thus, the probabilities that at least one point is missing from each pair are independent, so it is easy to the bound the probability that $\mathbf{k}$ is missing in $A+A$.  The same does not hold for differences, however, when a difference $\mathbf{k}\in L(nP) - L(nP)$ is small enough such that $\mathbf{x},\mathbf{x}+\mathbf{k}, \mathbf{x}+2\mathbf{k}\in L(nP)$ for some $\mathbf{x}\in L(nP)$.  Fortunately, as in \cite{MO}, the probability that such a small difference is missing is so tiny that a crude bound is sufficient.

\begin{lemma} \label{lemma:probMissingDiff}
Let $r>0$, and fix a fringe set $F_r\subset B_r(nP)$.  Let $A$ be chosen uniformly at random from all subsets $S\subset L(nP)$ such that $S\cap B_r(nP) = F_r$, and let $\mathbf{k} \in M_r(nP-nP)$ be large.  Then, for some constant $c > 0$ independent of $n$,
\begin{align}
\mathbb{P}[\mathbf{k}\not\in A-A]\ \le\ c\left(\frac{3}{4}\right)^{|L(nP\cap(nP-\mathbf{k})|/2}.
\end{align}
\end{lemma}

\begin{proof}
Define random variables $X_{\mathbf{j}}$ by setting $X_{\mathbf{j}} = 1$ if $\mathbf{j}\in A$ and $X_{\mathbf{j}} = 0$ otherwise.  We have $\mathbf{k}\not\in A-A$ if and only if $X_{\mathbf{j}}X_{\mathbf{j}+\mathbf{k}} = 0$ for all $\mathbf{j}\in L(nP \cap (nP-\mathbf{k}))$.

First suppose $\mathbf{k}$ is small such that $\mathbf{k}\in \frac{1}{2}(nP-nP)$, and suppose $\mathbf{k} = (k_1, k_2, \dots, k_D)$.  Define  
\begin{align}
G_n &\ := \  \left\{ (x_1, \dots, x_D) \in L(nP\cap (nP-\mathbf{k})) : \text{$\left\lfloor \frac{x_1}{k_1}\right\rfloor$ is even} \right\}, \\
H_n &\ := \  \left\{ (x_1, \dots, x_D) \in L(nP\cap (nP-\mathbf{k})) : \text{$\left\lfloor \frac{x_1}{k_1}\right\rfloor$ is odd} \right\},\\
J_n & \ :=\
\begin{cases}
G_n &\text{ if $|G_n| > |H_n|$} \\
H_n &\text{ if $|H_n| \ge |G_n|$.}
\end{cases}
\end{align}
It is possible that $J_n$ is $G_n$ or $H_n$ depending on $n$, hence the subscript notation.  It is clear that $\mathbf{x}+\mathbf{k}\not\in J_n$ for any $\mathbf{x}\in J_n$, and therefore the random variables $X_{\mathbf{j}}X_{\mathbf{j}+\mathbf{k}}$ are pairwise independent across all $\mathbf{j}\in J_n$.  As $|G_n| + |H_n| = |L(nP\cap (nP-\mathbf{k}))|$, it is further clear that
\begin{align}
|J_n|\ \ge\ \frac{1}{2}|L(nP\cap (nP-\mathbf{k}))|.
\end{align}

Thus, if our fixed fringe $F_r$ is missing exactly $l$ points, then we have that
\begin{align}
\mathbb{P}[X_{\mathbf{j}}X_{\mathbf{j}+\mathbf{k}} = 0 \text{ for all $\mathbf{j}\in L(nP \cap (nP-\mathbf{k}))$}]\ &\le\ \mathbb{P}[X_{\mathbf{j}}X_{\mathbf{j}+\mathbf{k}} = 0 \text{ for all $\mathbf{j}\in J_n$}] \nonumber \\
&=\ \prod_{\mathbf{j}\in J_n} \mathbb{P}[X_{\mathbf{j}}X_{\mathbf{j}+\mathbf{k}} = 0] \nonumber \\
&\le\ \left(\frac{3}{4}\right)^{|J_n|-l} \nonumber \\
&\le\ \left(\frac{3}{4}\right)^{|L(nP\cap(nP-\mathbf{k})|/2-l}.
\end{align}

Now suppose $\mathbf{k}\not\in\frac{1}{2}(nP-nP)$.  Then there exists no $\mathbf{j}\in L(nP)$ such that $\mathbf{j}, \mathbf{j}+\mathbf{k}, \mathbf{j}+2\mathbf{k} \in L(nP)$.  That is, the random variables $X_{\mathbf{j}}X_{\mathbf{j}+\mathbf{k}}$ are pairwise independent across all $\mathbf{j}\in L(nP \cap (nP-\mathbf{k}))$.  Then
\begin{align}
\mathbb{P}[X_{\mathbf{j}}X_{\mathbf{j}+\mathbf{k}} = 0 \text{ for all $\mathbf{j}\in L(nP \cap (nP-\mathbf{k}))$}]\ &=\ \prod_{\mathbf{j}\in L(nP \cap (nP-\mathbf{k}))} \mathbb{P}[X_{\mathbf{j}}X_{\mathbf{j}+\mathbf{k}} = 0] \nonumber \\
&\le\ \left(\frac{3}{4}\right)^{|L(nP\cap(nP-\mathbf{k}))|/2-l}.
\end{align}

In both cases, we can set $c = (3/4)^{-l}$ and the lemma follows.
\end{proof}


We define the regions $I$ and $C$ in the same way as in the sumset case.  In the difference set case, we analyze $nP \cap (nP-\kk)$ where in the sumset case we analyzed $nP \cap (\kk - nP)$.

The other aspect of the difference set case that deserves discussion is the difference set analogue of Lemma \ref{lemma:tlogballs}---that there exists a constant $t>0$ such that $I$ is contained in the union of balls of radius $t\log{n}$ around the vertices of $nP-nP$ for all $n$.  The reason the same proof carries through is that in any locally point symmetric polytope $Q$, the only uniquely formed differences are the differences between strictly antipodal vertices, and these differences are in one-to-one correspondence with the vertices of $Q-Q$.  When a difference $\kk \in Q-Q$ is close to one of these uniquely formed differences, $Q \cap (Q - \kk)$ is a parallelotope due to local point symmetry.  See Figure \ref{fig:diffsetk} for an illustration.

\begin{figure}
\centering
\begin{tikzpicture}[line cap=round,line join=round,>=triangle 45,x=1.0cm,y=1.0cm,scale=.9]

\draw
(0.3991342952275247,4.62266370699223)--
(-0.14194163367062618,3.109485261768588)--
(1.1302918316468693,2.22)--
(4.345530564652811,2.22)--
(4.564424451722249,2.8321608706179195)--
(2.0034628190898998,4.62266370699223)--
(0.3991342952275247,4.62266370699223);
\draw (3.1642266670997095,2.580678222170955) node {\small $Q - \mathbf{k}$};

\draw
(1.1466557706551417,6.274164641076495)--
(0.6055798417569908,4.760986195852852)--
(1.8778133070744862,3.8715009340842648)--
(5.093052040080427,3.8715009340842648)--
(5.311945927149866,4.483661804702184)--
(2.7509842945175165,6.274164641076495)--
(1.1466557706551417,6.274164641076495);
\draw (4.181122092214427,4.232179156255217) node {\small $Q$};

\draw[->]
(2.142700056336263,4.0628052065681155)--
(2.8902215318,5.7143061406);
\draw (2.275315836,4.99278) node {\small $\mathbf{k}$};
\draw [fill=black] (2.8902215318,5.7143061406) circle (1.5pt);
\draw [fill=black] (2.142700056336263,4.0628052065681155) circle (1.5pt);

\draw [fill=black] (1.8778133070744862,3.8715009340842648) circle (1.5pt);
\draw [fill=black] (2.7509842945175165,6.274164641076495) circle (1.5pt);
\draw (2.85,6.55) node {\small $\mathbf{v}$};
\draw (1.75,3.5) node {\small $\mathbf{u}$};

\end{tikzpicture}
\caption{\small \label{fig:diffsetk} A locally point symmetric polytope $Q$ with strictly antipodal vertices $\mathbf{u}$ and $\mathbf{v}$, a difference vector $\kk\in Q-Q$, and the translated polytope $Q-\kk$.  The difference $\kk$ is close to the uniquely formed difference $\mathbf{v}-\mathbf{u}$, and hence the intersection polytope $Q\cap (Q-\kk)$ is a parallelotope.}
\end{figure}
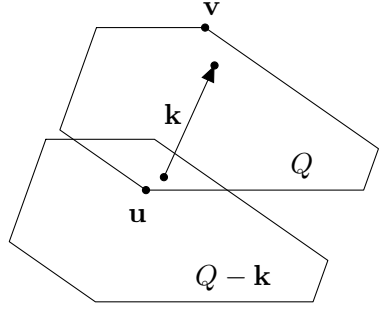


To be precise, in Step 1 of the proof for the sumset case, we show that $P \cap (\kk' - P)$ is just one point, and immediately conclude that $\kk'$ is a vertex of $P+P$.  Making the same conclusion takes some more work in the difference set case, which we do in the following lemma.

\begin{lemma} \label{lemma:verticesofdiffset}
Let $Q$ be a locally point symmetric polytope, and let $\kk \in Q-Q$.  The following statements are equivalent:
\begin{enumerate}
\item[(i)] $Q\cap (Q-\kk)$ consists of a single point, i.e., $\kk$ is a uniquely formed difference in $Q-Q$.
\item[(ii)] $\kk$ is a vertex of the polytope $Q-Q$.
\item[(iii)] $\kk = \mathbf{u}-\mathbf{v}$ for strictly antipodal vertices $\mathbf{u}$ and $\mathbf{v}$ of $Q$.
\end{enumerate}
\end{lemma}

\begin{proof}
In the proof of this lemma, we use the following facts about supporting cones which are not hard to prove:
\begin{enumerate}
\item Vertices $\mathbf{v}_i$ and $\mathbf{v}_j$ are strictly antipodal if and only if $C(\mathbf{v}_i) - \mathbf{v}_i$ and $C(\mathbf{v}_j) - \mathbf{v}_j$ intersect only at the origin.
\item Let $A$ be a face of $Q$ of dimension $k$ (a $k$-face), and let $\mathbf{a}$ be given in the interior of $A$ (relative to the $k$-dimensional affine space containing $A$).  If $\mathbf{x}$ is in the supporting cone of some vertex of $A$, then $\mathbf{a} + \epsilon \mathbf{x} \in Q$ for $\epsilon$ small.
\end{enumerate}
First, we show (i) $\implies$ (ii).  Let $\mathbf{u}$ and $\mathbf{v}$ be the unique points in $Q$ that satisfy $\kk = \mathbf{u} - \mathbf{v}$.  Suppose that $\mathbf{u}$ lies in a $k$-face $A$ and $\mathbf{v}$ lies in an $l$-face $B$, where $A$ and $B$ are chosen so that $k$ and $l$ are minimal. Note that $k$ and $l$ may range from $0$ to $D$.

Suppose there exist vertices $\mathbf{a} \in A$ and $\mathbf{b} \in B$ that are not strictly antipodal.  Then $C(\mathbf{a}) - \mathbf{a}$ and $C(\mathbf{b})-\mathbf{b}$ have an intersection containing some nonzero vector $\mathbf{x}$. Then for $\epsilon$ small,  $\mathbf{a}+\epsilon\mathbf{x}$ and $\mathbf{b}+\epsilon\mathbf{x}$ both lie in $Q$, so $\mathbf{u}+\epsilon\mathbf{x}$ and $\mathbf{v}+\epsilon\mathbf{x}$ both lie in $Q$. Thus, the difference $\kk = \mathbf{u} - \mathbf{v}$ is not uniquely formed.

Thus every vertex in $A$ must be strictly antipodal to every vertex in $B$.  Since the vertices of $Q$ are partitioned into strictly antipodal pairs, $A$ and $B$ must both be 0-faces; that is, $A = \{\mathbf{u}\}$ and $B = \{\mathbf{v}\}$, and $\mathbf{u}$ and $\mathbf{v}$ are strictly antipodal vertices as desired.

Next, we show (ii) $\implies$ (iii).  Let $\mathbf{u},\mathbf{v}\in Q$ such that $\kk = \mathbf{u}-\mathbf{v}$.  Observe that a point $\mathbf{v}\in Q$ is a vertex of $Q$ if and only if, for any non-zero translation vector $\mathbf{t}$, $\mathbf{v}+\mathbf{t}\in Q$ implies $\mathbf{v}-\mathbf{t}\not\in Q$.  The same statement holds for the polytope $Q-Q$.  If $\mathbf{u}$ is not a vertex of $Q$, then we have that $\mathbf{u}+\mathbf{t},\mathbf{u}-\mathbf{t}\in Q$ for some non-zero $\mathbf{t}$.  But then this implies that $\mathbf{k}+\mathbf{t} = (\mathbf{u}+\mathbf{t})-\mathbf{v}$ and $\mathbf{k}-\mathbf{t} = (\mathbf{u}-\mathbf{t})-\mathbf{v}$ are both contained in $Q-Q$, which contradicts that $\kk$ is a vertex of $Q-Q$.  Applying the same argument to $\mathbf{v}$, we get that $\mathbf{u}$ and $\mathbf{v}$ must both be vertices of $Q$.

Now suppose $\mathbf{u}$ and $\mathbf{v}$ are not strictly antipodal vertices.  We show that $\mathbf{k}+\mathbf{t},\mathbf{k}-\mathbf{t}\in Q-Q$ for some non-zero $\mathbf{t}$, which contradicts that $\mathbf{k}$ is a vertex of $Q-Q$.  Let $\mathbf{v}'$ denote the unique vertex that is strictly antipodal with $\mathbf{v}$.  For some small $\epsilon>0$, define $\mathbf{t} = \epsilon(\mathbf{v}'-\mathbf{u})$.  Clearly, $\mathbf{u}+\mathbf{t}\in Q$, so $\mathbf{k}+\mathbf{t} = (\mathbf{u}+\mathbf{t})-\mathbf{v}$ is contained in $Q-Q$.  Now consider the point $\mathbf{v}+\mathbf{t}$.  If $\mathbf{v}+\mathbf{t}\not\in Q$, and thus the half-line formed by extending out $\mathbf{t}$ from $\mathbf{v}$ is not contained in the supporting cone $C(\mathbf{v})$, then as $Q$ is locally point symmetric the parallel half-line formed by extending out from $\mathbf{v}'$ the vector $\mathbf{u}-\mathbf{v}'$ is not contained in $C(\mathbf{v}')$---a contradiction.  Thus $\mathbf{k}-\mathbf{t} = \mathbf{u}-(\mathbf{v}+\mathbf{t})$ is contained in $Q-Q$, which also forms a contradiction.  Thus, $\mathbf{u}$ and $\mathbf{v}$ must be strictly antipodal vertices.

Finally, we show (iii) $\implies$ (i).  Let $H_1$ and $H_2$ be parallel supporting hyperplanes meeting $Q$ at $\{\mathbf{u}\}$ and $\{\mathbf{v}\}$, respectively.  If $\mathbf{u}' \ne \mathbf{u}$ is a point in $Q$, then $\mathbf{u}' - \mathbf{k}$ lies on the side of $H_2$ opposite $Q$, so it cannot lie in $Q$.  Thus $\mathbf{k}$ is a uniquely formed difference in $Q$.
\end{proof}

The rest of the proof of Proposition \ref{prop:middleSums} carries over to the difference set case with trivial modifications.
\end{proof}




\section{Proof of Theorem \ref{thm:1}}

We begin by showing that the proportion $\rho^{s,d}_n$ of subsets $A\subset L(nP)$ missing exactly $s$ sums and exactly $2d$ differences approaches $0$ if $P$ is not locally point symmetric.  By Lemma \ref{lemma:uniqueDiff}, there exists a vertex $n\mathbf{v}$ and an edge $nE$ of $nP$ such that for all $\mathbf{e}\in nE$, the difference vectors $\mathbf{k} = \mathbf{e}-n\mathbf{v}$ and $-\mathbf{k}$ are uniquely formed.  Recall that $nE$ has at least $n+1$ lattice points.  Then, if $A$ is missing exactly $2d$ differences, at least $n+1-d$ of the lattice points in $nE$ must be present.  Thus, for $n>2d-1$, we see that
\begin{align}
\rho_n^{s,d}\ \le\ {n+1 \choose d}\left(\frac{1}{2}\right)^{n+1-d} \ =\ \Theta\left(\frac{n^d}{2^n}\right),
\end{align}
which approaches $0$ as $n\to\infty$.

We now handle the main case when $P$ is locally point symmetric.  For some radius $r$, we aim to construct a fringe set $F_r \subset B_r(nP)$ such that, for all sets $A$ that satisfy $A\cap B_r(nP) = F_r$,
\begin{align}
B_r(nP+nP) \setminus (A+A) &\ =\ s, \label{eq:fringeMissingSums}\\
B_r(nP-nP) \setminus (A-A) &\ =\  2d. \label{eq:fringeMissingDiffs}
\end{align}

If $P$ is 1-dimensional (a line segment), we simply place appropriate fringe sets at its ends as in \cite{He}.  Now suppose $P$ is $m$-dimensional for $m\ge 2$.  We can take a pair of strictly antipodal vertices $n\mathbf{v}_1$ and $n\mathbf{v}_2$ of $nP$, and a pair of parallel edges $nE_1$ and $nE_2$ such that $n\mathbf{v}_1 \in nE_1$ and $n\mathbf{v}_2\in nE_2$.  Suppose $nE_1$ and $nE_2$ contain $nb_{E_1}+1$ and $nb_{E_2}+1$ lattice points, respectively.  As discussed in the beginning of Section \ref{sec:edgeElem}, there exist injective affine transformations $T_{nE_1},T_{nE_2} : \mathbb{R}\to\mathbb{R}^D$ that form one-to-one correspondences between $[0,nb_{E_1}]$ and $L(nE_1)$, and between $[0,nb_{E_2}]$ and $L(nE_2)$, respectively.  We can also specify that $T_{nE_1}(0) = n\mathbf{v}_1$ and $T_{nE_2}(nb_{E_2}) = n\mathbf{v}_2$.  It is easily seen that $T_{nE_1}$ and $T_{nE_2}$ have the same associated linear transformation.

As shown in the proof of Theorem 8 in \cite{He}, for some $r'>0$ and $n> 2r'$, there exist sets $L_{s,d} \subset [0,r']$ and $U_{s,d} \subset [nb_{E_2}-r',nb_{E_2}]$ such that
\begin{align} \label{eq:hegartyMissingSums}
|[0,r'] \setminus (L_{s,d} + L_{s,d})| + |[2nb_{E_2}-r',2nb_{E_2}] \setminus (U_{s,d}+U_{s,d})|\ =\ s
\end{align}
and
\begin{align} \label{eq:hegartyMissingDiffs}
|[nb_{E_2}-r',nb_{E_2}] \setminus (U_{s,d}-L_{s,d})|\ =\ d.
\end{align}
Now define $r = \max\{r^+, r^-, r'\}$, where $r^+$ and $r^-$ are the constants given by Propositions \ref{prop:middleSums} and \ref{prop:middleDiffs}.  Let $B'_r(nP)$ denote the set
\begin{align}
B_r(nP)\setminus (T_{nE_1}([0,r']) \cup T_{nE_2}([nb_{E_2}-r',nb_{E_2}])),
\end{align}
and set
\begin{align}
F_r\ :=\ T_{nE_1}(L_{s,d}) \cup T_{nE_2}(U_{s,d}) \cup B'_r(nP).
\end{align}
That is, we place $L_{s,d}$ on one end of $nE_1$ and $U_{s,d}$ on one end of $nE_2$, and fill in all other points of $B_r(nP)$.  See Figure \ref{fig:lpsFringe} for an illustration.

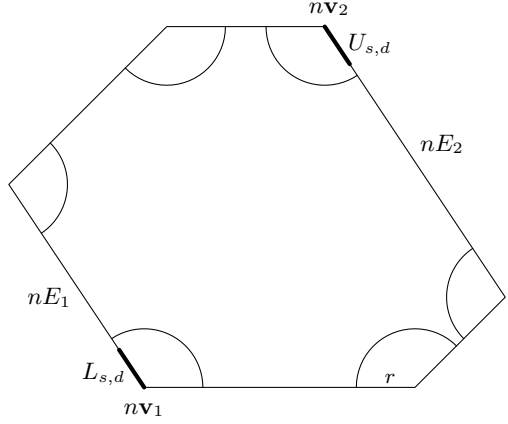
\begin{figure}[h]
\centering
\definecolor{ffffff}{rgb}{1,1,1}
\begin{tikzpicture}[line cap=round,line join=round,>=triangle 45,x=1.0cm,y=1.0cm,scale=0.6]

\draw (1.5,3) -- (5,3) -- (9,-3) -- (7,-5) -- (1,-5) -- (-2,-.5) -- cycle;

\draw [shift={(5.0,3.0)}] plot[domain=3.141592653589793:5.3003915839322575,variable=\t]({1.0*1.3*cos(\t r)+-0.0*1.3*sin(\t r)},{0.0*1.3*cos(\t r)+1.0*1.3*sin(\t r)});
\draw [shift={(1.5,3.0)}] plot[domain=-2.356194490192345:0.0,variable=\t]({1.0*1.3*cos(\t r)+-0.0*1.3*sin(\t r)},{0.0*1.3*cos(\t r)+1.0*1.3*sin(\t r)});
\draw [shift={(-2.0,-0.5)}] plot[domain=-0.9827937232473296:0.7853981633974483,variable=\t]({1.0*1.3*cos(\t r)+-0.0*1.3*sin(\t r)},{0.0*1.3*cos(\t r)+1.0*1.3*sin(\t r)});
\draw [shift={(1.0,-5.0)}] plot[domain=0.0:2.1587989303424644,variable=\t]({1.0*1.3*cos(\t r)+-0.0*1.3*sin(\t r)},{0.0*1.3*cos(\t r)+1.0*1.3*sin(\t r)});
\draw [shift={(7.0,-5.0)}] plot[domain=0.7853981633974483:3.141592653589793,variable=\t]({1.0*1.3*cos(\t r)+-0.0*1.3*sin(\t r)},{0.0*1.3*cos(\t r)+1.0*1.3*sin(\t r)});
\draw [shift={(9.0,-3.0)}] plot[domain=2.1587989303424635:3.926990816987242,variable=\t]({1.0*1.3*cos(\t r)+-0.0*1.3*sin(\t r)},{0.0*1.3*cos(\t r)+1.0*1.3*sin(\t r)});

\draw[ultra thick] (1,-5) -- (0.44529980377477085,-4.167949705662156);
\draw[ultra thick] (5,3) -- (5.554700196225229,2.1679497056621564);

\begin{scriptsize}
\draw (5.1,3.4) node {$n\mathbf{v}_2$};
\draw (1,-5.5) node {$n\mathbf{v}_1$};
\draw (6.45,-4.8) node {\tiny $r$};
\draw (0.1,-4.7) node {$L_{s,d}$};
\draw (6, 2.6) node {$U_{s,d}$};
\draw (-1.1,-3) node {$nE_1$};
\draw (7.6,.4) node {$nE_2$};
\end{scriptsize}
\end{tikzpicture}
\caption{\small \label{fig:lpsFringe} A locally point symmetric hexagon.  The fringe set $F_r$ lives within the balls of radius $r$ centered about the vertices.  The one-dimensional fringe sets $L_{s,d}$ and $U_{s,d}$ are placed on corresponding parallel edges $nE_1$ and $nE_2$ of the strictly antipodal vertices $n\mathbf{v}_1$ and $n\mathbf{v}_2$.}
\end{figure}

Now let $A$ be uniformly randomly chosen from all subsets $S\subset L(nP)$ such that $S\cap B_r(nP) = F_r$.  We see that
\begin{align}
A \cap nE_1 \ = \  T_{nE_1}(S_1), \ \ \ \ \ A \cap nE_2 \ = \  T_{nE_2}(S_2),
\end{align}
for some sets $S_1, S_2 \in \mathbb{Z}$ such that $S_1 \cap [0,r'] = L_{s,d}$ and $S_2 \cap [nb_{E_2}-r',nb_{E_2}] = U_{s,d}$.  By Lemma \ref{lemma:edgeSumSet}, there exist injective affine transformations $T_{nE_1+nE_1}, T_{nE_2+nE_2}: \mathbb{R}\to\mathbb{R}^D$ such that
\begin{align}
(A+A) \cap (nE_1+nE_1) &\ = \  T_{nE_1+nE_1}(S_1 + S_1), \nonumber\\
(A+A) \cap (nE_2+nE_2) &\ = \  T_{nE_2+nE_2}(S_2 + S_2).
\end{align}
It is easy to show that, then,
\begin{align}
(A+A) \cap T_{nE_1+nE_1}([0,r']) &\ = \  T_{nE_1+nE_1}((S_1+S_1)\cap [0,r']) \nonumber \\
						    &\ = \  T_{nE_1+nE_1}((L_{s,d}+L_{s,d}) \cap [0,r']),
\end{align}
and similarly
\begin{align}
(A+A) \cap T_{nE_2+nE_2}([2nb_{E_2}-r',2nb_{E_2}])\ =\ T_{nE_2+nE_2}((U_{s,d}+U_{s,d}) \cap [2nb_{E_2}-r',2nb_{E_2}]).
\end{align}
It follows from \eqref{eq:hegartyMissingSums} that $A+A$ is missing a total of exactly $s$ sums in the regions $T_{nE_1+nE_1}([0,r'])$ and $T_{nE_2+nE_2}([nb_{E_2}-r',nb_{E_2}])$.

Similarly, by Lemma \ref{lemma:edgeDiffSet}, there exists an injective affine transformation $T_{nE_2-nE_1}: \mathbb{R}\to \mathbb{R}^D$ such that
\begin{align}
(A-A) \cap (nE_2 - nE_1)\ =\ T_{nE_2 - nE_1}(S_2-S_1),
\end{align}
and we can show that
\begin{align}
(A-A) \cap T_{nE_2-nE_1}([nb_{E_2}-r',nb_{E_2}])\ =\ T_{nE_2-nE_1}((U_{s,d}-L_{s,d}) \cap [nb_{E_2}-r',nb_{E_2}]).
\end{align}
It follows by \eqref{eq:hegartyMissingDiffs} that $A-A$ is missing exactly $2d$ differences in the regions $T_{nE_2-nE_1}([nb_{E_2}-r',nb_{E_2}])$ and $-T_{nE_2-nE_1}([nb_{E_2}-r',nb_{E_2}])$.

Finally, it is not hard to show that all other elements in $B_r(nP+nP)$ and $B_r(nP-nP)$ are present, that is,
\begin{align}
B_r(nP+nP) \setminus (T_{nE_1+nE_1}([0,r']) \cup T_{nE_2+nE_2}([2nb_{E_2}-r',2nb_{E_2}])) \subset A+A, \nonumber\\
B_r(nP-nP) \setminus (T_{nE_2-nE_1}([nb_{E_2}-r',nb_{E_2}]) \cup -T_{nE_2-nE_1}([nb_{E_2}-r',nb_{E_2}]))\subset A-A.
\end{align}
Thus, we satisfy \eqref{eq:fringeMissingSums} and \eqref{eq:fringeMissingDiffs}.

Let $p^+ > 1/2$ and $p^- > 1/2$.  By Propositions \ref{prop:middleSums} and \ref{prop:middleDiffs}, we have that $M_r(nP+nP) \subset A+A$ with probabilities at least $p^+$, and that $M_r(nP-nP) \subset A-A$ with probability at least $p^-$, where $p^+$ and $p^-$ are fixed independent of $n$.  It follows that $M_r(nP+nP) \subset A+A$ and $M_r(nP-nP) \subset A-A$ with positive probability independent of $n$.  Thus, a positive proportion of the subsets $A$, and thus a positive proportion of all subsets of $L(nP)$, have exactly $s$ missing sums and exactly $2d$ missing differences. \qed

\section{Proof of Theorem \ref{thm:2}}

Similarly to as in the proof of Theorem \ref{thm:1}, the main task is to construct a fringe set $F_r \subset B_r(nP)$ for some radius $r$ such that, for all sets $A$ that satisfy $A\cap B_r(nP) = F_r$,
\begin{align}
B_r(nP+nP) \setminus (A+A) \ = \  s, \ \ \ \ \ B_r(nP-nP) \setminus (A-A) \ \ge\ 2d.
\end{align}
Once we construct $F_r$, the proof concludes identically.  The difference here is that because we do not assume local point symmetry in $P$, we are no longer guaranteed the existence of `distant' parallel edges, and thus cannot use Lemma \ref{lemma:edgeDiffSet} to control the number of missing differences.  On the other hand, we do not need to limit the number of missing differences so long as there are at least $2d$ of them.  This allows us to use Lemma \ref{lemma:uniqueDiff} to our advantage.

If $P$ is locally point symmetric, then we simply construct $F_r$ as in the proof of Theorem \ref{thm:1}.  Now suppose $P$ is not locally point symmetric.  Let $n\mathbf{v}$ and $nE_1$ denote, respectively, the vertex and edge returned by Lemma \ref{lemma:uniqueDiff} when it is applied to $nP$, and let $nE_2$ denote some other edge of $nP$ that is distinct from $nE$.  If $nE_1$ and $nE_2$ contain, respectively, $nb_{E_1}+1$ and $nb_{E_2}+1$ lattice points, then let $T_{nE_1}, T_{nE_2}: \R\to\R^D$ denote the injective affine transformations that form one-to-one correspondences between $[0,nb_{E_1}]$ and $L(nE_1)$, and between $[0,nb_{E_2}]$ and $L(nE_2)$, respectively.

As shown in the proof of Theorem 8 in \cite{He}, for some $r'>0$ and $n> 2r'$, there exist sets $L_{s} \subset [0,r']$ and $U_{s} \subset [nb_{E_2}-r',nb_{E_2}]$ such that
\begin{align} \label{eq:hegartyMissingSums2}
|[0,r'] \setminus (L_{s,0} + L_{s,0})| + |[2nb_{E_2}-r',2nb_{E_2}] \setminus (U_{s,0}+U_{s,0})|\ =\ s.
\end{align}
Further define
\begin{align}
R_d\ :=\ [0,d-1] \cup [2d, 3d-1],
\end{align}
and observe that $[0,3d-1] \subset R_d + R_d$.

Define $r = \max\{r^+, r^-, r', 3d-1\}$, where $r^+$ and $r^-$ are the constants given by Propositions \ref{prop:middleSums} and \ref{prop:middleDiffs}, respectively.  Define
\begin{align}
B'_r(nP)\ :=\ B_r(nP) \setminus (T_{nE_1}([0,3d-1]) \cup T_{nE_2}([0,r']) \cup T_{nE_2}([nb_{E_2}-r', nb_{E_2}]) ),
\end{align}
and set
\begin{align}
F_r\ :=\ T_{nE_1}(R_d) \cup T_{nE_2}(L_{s}) \cup T_{nE_2}(U_{s}) \cup B'_r(nP).
\end{align}
That is, we place $R_d$ on one end of $nE_1$, $L_{s}$ on one end of $nE_2$, and $U_{s}$ on the other end of $nE_2$, and fill in all other points of $B_r(nP)$.  See Figure \ref{fig:nonlpsQuad} for an illustration.

\begin{figure}
\centering
\begin{tikzpicture}[line cap=round,line join=round,>=triangle 45,x=1.0cm,y=1.0cm]
\clip(-0.42,-0.46) rectangle (8.18,6.46);
\draw (0.56,1.98)-- (4.94,5.62)-- (6.88,3.4)-- (5.04,1.12)-- (0.56,1.98);
\draw [domain=-0.42:8.18] plot(\x,{(--21.87-2.22*\x)/1.94});
\draw [domain=-0.42:8.18] plot(\x,{(--5.08-2.22*\x)/1.94});

\draw [shift={(4.94,5.62)}] plot[domain=3.834982009741346:5.430580817999774,variable=\t]({1.0*1.3000000000000003*cos(\t r)+-0.0*1.3000000000000003*sin(\t r)},{0.0*1.3000000000000003*cos(\t r)+1.0*1.3000000000000003*sin(\t r)});
\draw [shift={(0.56,1.98)}] plot[domain=-0.1896571072717741:0.6933893561515527,variable=\t]({1.0*1.3*cos(\t r)+-0.0*1.3*sin(\t r)},{0.0*1.3*cos(\t r)+1.0*1.3*sin(\t r)});
\draw [shift={(5.04,1.12)}] plot[domain=0.8917910180963523:2.951935546318019,variable=\t]({1.0*1.3000000000000003*cos(\t r)+-0.0*1.3000000000000003*sin(\t r)},{0.0*1.3000000000000003*cos(\t r)+1.0*1.3000000000000003*sin(\t r)});
\draw [shift={(6.88,3.4)}] plot[domain=2.28898816440998:4.0333836716861455,variable=\t]({1.0*1.2999999999999996*cos(\t r)+-0.0*1.2999999999999996*sin(\t r)},{0.0*1.2999999999999996*cos(\t r)+1.0*1.2999999999999996*sin(\t r)});

\draw [ultra thick] (6.88,3.4) -- (6.42,3.93); 
\draw [ultra thick] (4.94,5.62) -- (4.26,5.06); 
\draw [ultra thick] (0.56,1.98) -- (1.33,2.62); 

\begin{scriptsize}
\draw (0.3,1.72) node {$n\mathbf{v}$};
\draw (6.3,4.7) node {$nE_1$};
\draw (2.5,4.16) node {$nE_2$};
\draw (0.8,2.6) node {$L_s$};
\draw (4.4,5.6) node {$U_s$};
\draw (7,3.8) node {$R_d$};
\draw (4.5,1.4) node {\tiny $r$};
\end{scriptsize}
\end{tikzpicture}
\caption{\small \label{fig:nonlpsQuad} A quadrilateral that is not locally point symmetric.  Again, the fringe set $F_r$ lives within the balls of radius $r$ centered about the vertices.  The one-dimensional fringe set $R_d$ is placed on edge $nE_1$, and sets $L_s$ and $U_s$ are placed on opposite ends of edge $nE_2$.}
\end{figure}
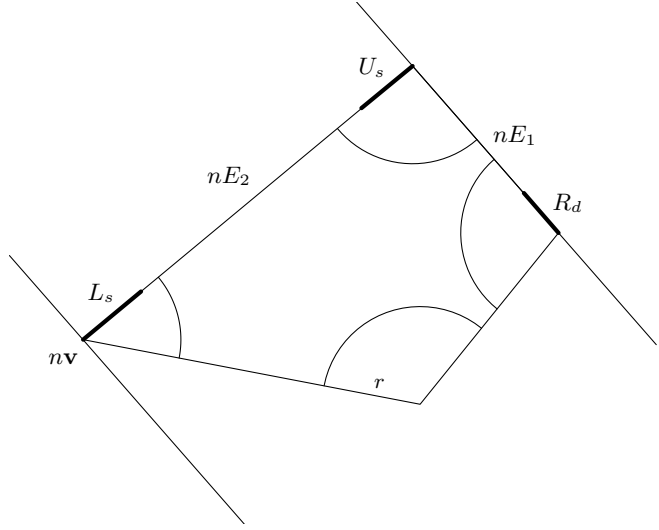

Now let $A$ be uniformly randomly chosen from all subsets $S\subset L(nP)$ satisfying $S\cap B_r(nP) = F_r$.  We see that
\begin{align}
A\cap nE_1 \ =\ T_{nE_1}(S_1), \ \ \ \ \ A\cap nE_2 \ = \ T_{nE_2}(S_2)
\end{align}
for some sets $S_1, S_2\in \Z$ such that $S_1\cap [0,3d-1] = R_d$, $S_2\cap [0,r'] = L_{s}$, and $S_2\cap [nb_{E_2}-r', nb_{E_2}] = U_{s}$.  By Lemma \ref{lemma:edgeSumSet}, there exists an injective affine transformation $T_{nE_2+nE_2}:\R\to\R^D$ such that
\begin{align}
(A+A) \cap (nE_2 + nE_2)\ =\ T_{nE_2+nE_2}(S_2+S_2).
\end{align}
It is easy to show that, then,
\begin{align}
(A+A) \cap T_{nE_2+nE_2}([0,r'])\ =\ T_{nE_2+nE_2}((L_{s}+L_{s})\cap[0,r'])
\end{align}
and
\begin{align}
(A+A) \cap T_{nE_2+nE_2}([2nb_{E_2}-r',2nb_{E_2}])\ =\ T_{nE_2+nE_2}((U_{s}+U_{s})\cap[2nb_{E_2}-r',2nb_{E_2}]).
\end{align}
It follows from \eqref{eq:hegartyMissingSums2} that $A+A$ is missing a total of exactly $s$ sums in the regions $T_{nE_2+nE_2}([0,r'])$ and $T_{nE_2+nE_2}([2nb_{E_2}-r',2nb_{E_2}])$.

As $A$ is missing $d$ lattice points along the edge $nE_1$, it follows from Lemma \ref{lemma:uniqueDiff} that $A-A$ is missing at least $2d$ differences.  Let $T_{nE_1+nE_1}:\R\to\R^D$ be the injective affine transformation returned by Lemma \ref{lemma:edgeSumSet} when applied to edge $nE_1$; because $[0,3d-1] \subset R_d + R_d$, we can show in a similar manner to the argument above that $A+A$ is not missing any sums in the region $T_{nE_1+nE_1}([0,3d-1])$.

Finally, it is not hard to show that all other elements in $B_r(nP+nP)$ are present.  That is, all points in the set
\begin{align}
B_r(nP+nP) \setminus (T_{nE_1+nE_1}([0,3d-1]) \cup T_{nE_2+nE_2}([0,r']) \cup T_{nE_2+nE_2}([2nb_{E_2}-r', 2nb_{E_2}]) )
\end{align}
are present in $A+A$.  The proof concludes identically as in the proof of Theorem \ref{thm:1} from here. \qed

\section{Future Directions} \label{sec:conclusion}

There are several natural directions in which to proceed from here.  
One conjecture is that the proportion $\rho_n^{s,d}$ converges if $P$ is locally point symmetric.  Zhao \cite{Z} proved this in the one-dimensional case, and with some work his arguments might be extended to arbitrary $D$-dimensional polytopes.  This would likely involve modifying Zhao's notion of a semi-rich set so that it is defined in terms of the supporting cone for each vertex of the polytope.

Another problem is to consider the proportion $\rho_n$ of MSTD subsets of $L(nP)$ for an arbitrary polytope $P$, which we discuss here in some detail.  As mentioned in the introduction, neither Theorem \ref{thm:1} nor Theorem \ref{thm:2} implies that $\rho_n$ is bounded below by a positive constant as $n\to\infty$.  This is due to the fact that $L(nP)$ is usually not balanced, such that $L(nP)-L(nP)$ is much larger than $L(nP)+L(nP)$.  In particular, the ratio $|L(nP)-L(nP)|/|L(nP)+L(nP)|$ is essentially constant as $n$ grows, and so we have that $|L(nP)-L(nP)|-|L(nP)+L(nP)|$ grows on the order of $n^D$.  Reformulating the problem in terms of missing sums and differences, we see that a subset $A\subset L(nP)$ must be missing $\sim n^D$ differences for it even possibly to be MSTD.

There are some factors that, upon first glance, suggest that there may be many such subsets.   If $P$ is not locally point symmetric (and therefore not balanced), then Lemma \ref{lemma:uniqueDiff} shows that there are many uniquely formed differences in $L(nP)-L(nP)$.  In other words, though the potential difference set is large in size, it is very fragile in that many of its differences are missing with high probability.  For example, consider the lattice points of the tetrahedron $nT$ in $\R^3$ determined by vertices $A = (-n,0,0)$, $B = (n,0,0)$, $C = (0,-n,n)$, and $D = (0,n,n)$.
By bounding $nT$ with supporting planes $z = 0$ and $z = n$, we see that any difference between a point in edge $\overline{AB}$ and a point in edge $\overline{CD}$ is uniquely formed.  Similarly, we have that any difference formed by $A$ and a point on the face $\triangle BCD$ is uniquely formed.  As this holds for any difference vector formed by a vertex and a point on the opposite face, or by points on skew edges of $nT$, we see that the presence of the boundary points of $nT$ have a significant impact on the size of the difference set---in this sense, the natural fringe extends to the entire boundary of $nT$ rather than being restricted to the balls centered about the vertices.

However, even if we make the strong imposition that a subset $A\subset L(nT)$ is missing all boundary points of $nT$, this still would not amount to the necessary $\sim n^3$ missing differences.  Each vertex forms around $\sim n^2$ uniquely formed differences with points on the opposite face, and each of the $\sim n$ points on the edges of $nT$ forms $\sim n$ unique differences with points on the opposite skew edge.  This suggests that subsets $A\subset L(nT)$ whose difference set is even within the range of the potential sumset become vanishingly rare as $n$ grows.

The tetrahedron is, in a sense, very far from being locally point symmetric.   The reason is that for each vertex $\mathbf{v}$, there are hyperplanes that support the tetrahedron precisely at $\mathbf{v}$ and the opposite face $F$.  Consider now the following locally point symmetric hexagon $H$, depicted in Figure \ref{fig:H}.
\begin{figure}
\centering
\begin{tikzpicture}
\draw [step=0.5,thin,gray!40] (-1.4,-.4) grid (3.4,4.4);
\draw [->] (-1.5,0) -- (3.5,0); 
\draw [->] (0,-.5) -- (0,4.5); 
\draw[thick] (0,0) -- (2,0) -- (3,2) -- (1,4) -- (0.5,4) -- (-1,1) -- (0,0);
\end{tikzpicture}
\caption{\small \label{fig:H} Locally point symmetric hexagon $H$}
\end{figure}
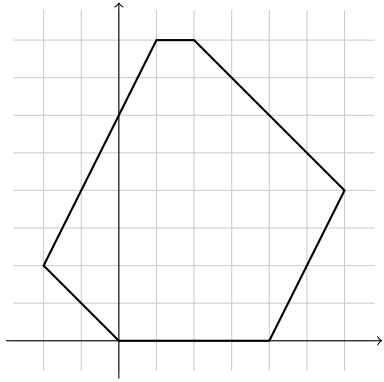
We can compute that $|L(H)+L(H)| = 181$ and $|L(H)-L(H)| = 187$, and the difference in these cardinalities grows quadratically as we take dilations of $H$.  In this case, however, the difference set is much more robust.  Because $H$ is locally point symmetric, we have no uniquely formed differences except those formed by pairs of strictly antipodal vertices.  Thus, we are forced to impose even stronger conditions on missing points in a subset $A\subset L(nH)$ for it to miss the required $\sim n^2$ differences.

From these considerations in combination with Corollary \ref{cor:3}, we make the following conjecture:

\begin{conjecture}
Let $P$ be polytope in $\R^D$ with vertices in $\Z^D$.  Then the proportion $\rho_n$ of MSTD subsets of $L(nP)$ approaches $0$ as $n\to\infty$ if and only if $L(P)$ is not balanced.
\end{conjecture}

This raises the question of how to characterize polytopes $P$ for which $L(P)$ is not balanced.  We know that if $P$ is point symmetric, then $L(P)$ is balanced,  but does the converse hold true?  Or perhaps there exists some $P$, locally point symmetric but not point symmetric, for which $L(P)$ is balanced.  Does this imply that $L(nP)$ is also balanced for all $n$?

Finally, it is interesting to examine how the limiting proportions of $\rho_n$ and $\rho_n^{s,d}$ (assuming they exist) change as we vary our polytope $P$.  For example, if $P$ is a rectangle in $\R^2$, how do they change as we vary the ratio of side lengths?  What happens as we increase the number of sides?  How do the limiting proportions change as we vary the dimension $D$?  Do $\rho_n$ and $\rho^{s,d}_n$ exhibit monotonic growth with the dilation factor $n$, as computations suggest when $P$ is an interval (see \cite{MO})?  We hope to investigate these questions theoretically and numerically in a future paper.

\appendix

\section{Number of Pairs of Strictly Antipodal Vertices} \label{appendix:A}

We show that Lemma \ref{lemma:saVertices} follows from the work of Nguy\^{e}n and Soltan \cite{NS}.  We restate Lemma \ref{lemma:saVertices} here for the reader's convenience.

\begin{lemma} \label{lemma:saVertices2}
Let $Q$ be a $D$-dimensional polytope with $m$ vertices in $\R^D$.  Then $Q$ is locally point symmetric if and only if $Q$ has exactly $m/2$ pairs of strictly antipodal vertices.
\end{lemma}

Let $s(Q)$ denote the number of pairs of strictly antipodal vertices in a convex polytope $Q$.  The following theorems come from Theorems 1 and 3 of \cite{NS}.

\begin{theorem} \label{thm:NS1}
For a convex polygon $Q\subset\R^2$ with $m$ vertices,
\begin{align}
s(Q)\ =\ m-k,
\end{align}
where $k$ ($0\le k\le \lfloor m/2 \rfloor$) is the number of pairs of parallel sides in $Q$.
\end{theorem}

\begin{theorem} \label{thm:NS3}
For a convex $D$-dimensional polytope $Q\subset\R^D$, $m\ge D+1$, $D\ge 3$,
\begin{align}
s(Q)\ \ge\ \lceil m/2 \rceil.
\end{align}
For an even $m$, the equality $s(Q) = \lceil m/2 \rceil$ holds if and only if $m\ge 2D$ and the vertices of $Q$ can be divided into $m/2$ pairs such that for each pair $\{\mathbf{u},\mathbf{v}\}$,
\begin{align}
C(\mathbf{u}) - \mathbf{u}\ =\ \mathbf{v} - C(\mathbf{v}).
\end{align}
For an odd $m$, the equality $s(Q) = \lceil m/2 \rceil$ holds if and only if $m\ge 4D-1$ and some $(m-3)/2$ pairwise disjoint subsets of the form $\{\mathbf{u},\mathbf{v}\}$ can be chosen from the vertex set such that
\begin{align}
C(\mathbf{u}) - \mathbf{u}\ =\ \mathbf{v} - C(\mathbf{v})
\end{align}
for each of them, and the remaining three vertices $\mathbf{x}, \mathbf{y}, \mathbf{z}$ satisfy the relation
\begin{align}
(C(\mathbf{x}) - \mathbf{x}) \cap (C(\mathbf{y}) - \mathbf{y})\ =\ \mathbf{z} - C(\mathbf{z}).
\end{align}
\end{theorem}

Let $Q$ be a $D$-dimensional polytope with $m$ vertices in $\R^D$.  If $D = 1$, then $Q$ is an interval and satisfies Lemma \ref{lemma:saVertices2}.  If $D \ge 3$, then Lemma \ref{lemma:saVertices2} follows immediately from Theorem \ref{thm:NS3}.

It remains to show Lemma \ref{lemma:saVertices2} in the case $D = 2$.  By Theorem \ref{thm:NS1}, it suffices to show that $Q$ is locally point symmetric if and only if $Q$ has exactly $m/2$ pairs of parallel sides.  As showing this is easy, we sketch the idea here.  The forward implication is immediate.  Now suppose $Q$ has exactly $m/2$ pairs of parallel sides, and further suppose $Q$ has vertices $\mathbf{v}_1, \mathbf{v}_2, \dots, \mathbf{v}_m$ in clockwise order.  We can show that for any pair of parallel sides $E = \overline{\mathbf{v}_i\mathbf{v}_{i+1}}$ and $F = \overline{\mathbf{v}_j\mathbf{v}_{j+1}}$ of $Q$, there exist supporting lines $L_1$ and $L_2$ such that $L_1 \cap Q = E$ and $L_2\cap Q = F$.  From there, we can show that $\mathbf{v}_i$ and $\mathbf{v}_j$ are strictly antipodal, and $\mathbf{v}_{i+1}$ and $\mathbf{v}_{j+1}$ are strictly antipodal.  That $Q$ is locally point symmetric follows easily from there.


\end{document}